\documentclass[twoside,12pt]{article}

\usepackage{amssymb, amsmath}
\usepackage{mathtools}
\usepackage[thmmarks,amsmath]{ntheorem}

\usepackage[OT1]{fontenc}

\usepackage[utf8]{inputenc}
\usepackage[T1]{fontenc}
\usepackage{dsfont}
\usepackage{latexsym} 

\usepackage{enumerate}	
\usepackage{footnpag} 
\usepackage{makeidx}  

\usepackage{geometry}	
\usepackage{parcolumns}

\usepackage{graphicx}
\usepackage{color} 

\usepackage[all]{xy}
\usepackage{url}

\usepackage{mdwlist} 

\newcommand{\Lim}[1]{\raisebox{0.5ex}{\scalebox{0.8}{$\displaystyle \lim_{#1}\;$}}}

\DeclareTextFontCommand{\textcyr}{\cyr}

\def\Span{\mathop{\textnormal{span}\hspace*{-.7mm}}}   
\renewcommand{\geq}{\geqslant}	
\renewcommand{\leq}{\leqslant}
\def\mid|{\hs\middle|\hs}
\def\hs{\hspace*{0.1cm}}

\def\1{\mathbb{1}}
\newcommand{\sub}{\subseteq}	
\renewcommand{\1}{\mathds{1}}					

\newcommand{\B}{\mathcal{B}}

\newcommand{\I}{\mathcal{I}}

\newcommand{\NN}{\mathbb{N}}	
\newcommand{\RR}{\mathbb{R}} 
\newcommand{\ZZ}{\mathbb{Z}} 

\def\t{\textnormal}

\newcommand{\leer}{\varnothing}
\newcommand{\ohne}{\backslash}

\def\c{\cite }
\newcommand{\hide}[1]{} 

\def\eps{\varepsilon}
\def\phi{\varphi}
\def\rho{\varrho}

\DeclareMathAlphabet\mathbfcal{OMS}{cmsy}{b}{n}

\def\dd{^\textnormal{dd}}
\newcommand{\Sup}{\vee}		
\newcommand{\Inf}{\wedge}	%

\let\int\relax 
\DeclareMathOperator{\int}{int}

\renewcommand{\theequation}{\arabic{equation}}

\usepackage{enumitem}
\setenumerate[0]{label=(\roman*)}
\setenumerate[2]{label=(\alph*)}

\newcommand*{\smallmath}[2][4]{\scalebox{#1}{$#2$}}%
\newcounter{Zaehler}
\theoremstyle{plain}
\theoremheaderfont{\normalfont\bfseries}
\theorembodyfont{\itshape}
\theoremseparator{.} 
\theorempreskip{\topsep}
\theorempostskip{\topsep} 
\theoremnumbering{arabic}
\theoremsymbol{}

\newtheorem{theorem}{Theorem}
\newtheorem{lemma}[theorem]{Lemma}
\newtheorem{corollary}[theorem]{Corollary}
\newtheorem{proposition}[theorem]{Proposition}
\newtheorem{definition}[theorem]{Definition}			
\theoremstyle{plain}
\theoremheaderfont{\normalfont\bfseries}
\theorembodyfont{\upshape}
\theoremseparator{.} 
\theorempreskip{\topsep}
\theorempostskip{\topsep} 
\theoremnumbering{arabic}
\theoremsymbol{}

\newtheorem{example}[theorem]{Example}
\newtheorem*{remark}{Remark}

\makeatletter
\newcommand{\eqnum}{\leavevmode\hfill\refstepcounter{equation}\textup{\tagform@{\theequation}}}
\makeatother

\theoremstyle{nonumberplain}
\theoremheaderfont{\itshape}
\theorembodyfont{\normalfont}
\theoremseparator{.} 
\theorempreskip{\topsep}
\theorempostskip{\topsep} 
\theoremsymbol{\ensuremath\square} 
\qedsymbol{$\Box$}

\newtheorem{proof}{Proof}

\theoremstyle{empty}   
\theorembodyfont{\itshape}
\theoremseparator{.} 
\theorempreskip{\topsep}
\theorempostskip{\topsep} 
\begin{document}
\title{Projection bands and atoms in pervasive\\ pre-Riesz spaces}
\author{Anke Kalauch\footnote{FR Mathematik, Institut für Analysis, TU Dresden, 01062 Dresden, Germany,\newline \texttt{anke.kalauch@tu-dresden.de}},
Helena Malinowski\footnote{Unit for BMI, North-West University, Private Bag X6001, Potchefstroom, 2520, South Africa,\newline \texttt{lenamalinowski@gmx.de}}}
\date{\today}
\maketitle
\begin{abstract}
In vector lattices, the concept of a projection band is a basic tool. We deal with projection bands in the more general setting of an Archimedean pre-Riesz space $X$. We relate them to projection bands in a vector lattice cover $Y$ of $X$.  If $X$ is pervasive, then a projection band in $X$ extends to a projection band in $Y$, whereas the restriction of a projection band $B$ in $Y$ is not a projection band in $X$, in general. We give conditions under which the restriction of $B$ is a projection band in $X$.
We introduce atoms and discrete elements in $X$ and show that every atom is discrete. The converse implication is true, provided $X$ is pervasive. 
In this setting,  
we link atoms in $X$ to atoms in $Y$. 
If 
$X$ contains an atom $a>0$, we show that the principal band generated by $a$ is a projection band.
Using atoms in a finite dimensional Archimedean pre-Riesz space $X$, we establish that $X$ is pervasive if and only if it is a vector lattice. 
\end{abstract}




\section{Introduction}
Projection bands and atoms are both fundamental concepts in the vector lattice theory.
For an Archimedean vector lattice $Y$, a band $B$ in $Y$ is called a \emph{projection band} if $Y=B\oplus B^{\t{d}}$, where $B^{\t{d}}$ is the disjoint complement of $B$. The band $B$ is a projection band if and only if there exists an order projection onto $B$, i.e.\ a positive linear operator $P\colon Y\rightarrow Y$ with $P^2=P$ and $P(Y)= B$. 
 Moreover, for every $y\in Y_+$ we have $0\leq P(y) \leq y$, see, e.g.,  \c[p.~133~ff.]{Zaa1}.
 Projections bands can similarly be introduced in pre-Riesz spaces using the notion of a band in \c{1}. Pre-Riesz spaces are precisely those partially ordered vector spaces that can be order densely embedded into vector lattices, their so-called vector lattice covers; see \cite{vanHaa}. In \cite{14} it is shown that the relation between projection bands and order projections in pre-Riesz spaces is similar to the one in vector lattices. 
 
 Atoms play an important role in the investigation of cones in vector lattices and finite dimensional ordered vector spaces, see, e.g., \c[Section~1.6]{AliTou} and \c[Section~26]{Zaa1}.
We call a strictly positive element $a$ of an ordered vector space $X$ an \emph{atom}, if for every $x\in X$ with $0\leq x<a$ there is $\lambda\in\RR$ with $x=\lambda a$.
In an Archimedean vector lattice $Y$ a positive element $a$ is an atom if and only if $a$ is a discrete element, i.e.\ for every pair of disjoint elements $u,v\in Y$ with $0\leq u\leq a$ and $0\leq v\leq a$ it follows $u=0$ or $v=0$. 
The notion of an atom as well as the notion of a discrete element can be 
generalized to pre-Riesz spaces, since one can define disjointness in pre-Riesz spaces according to \c{1}. 

For the investigation of structures in pre-Riesz spaces,  
the approach to use vector lattice covers and the restriction and extention method described in \cite[Section 2.8]{PRS} turned out to be fruitful. The techniques used in the present paper follow the same spirit.
We deal with the basic question how projection bands in a pre-Riesz space are related to projection bands in a corresponding vector lattice cover. Furthermore, we investigate under which conditions atoms and discrete elements in a pre-Riesz space coincide. We also study the problem how atoms in a pre-Riesz space and atoms in a corresponding vector lattice cover are linked.

In the vector lattice theory, the following statement is well-known, see \c[Theorem~26.4]{Zaa1}.
\begin{theorem}\label{LuxZahn}
Let $Y$ be an Archimedean vector lattice. If $a \in Y$ is an atom, then the principal band $\B_a$ generated by $a$ consists of all real multiples of $a$, and $\B_a$ is a projection band.
\end{theorem}
Theorem~\ref{LuxZahn} states that $Y$ admits the decomposition $Y=\B_a\oplus\B_a^{\t{d}}$. We investigate under which conditions  a similar statement is valid in pre-Riesz spaces.
It turns out that pervasive pre-Riesz spaces play a crucial role.

The paper is organized as follows.
In Section~\ref{preliminaries} all preliminaries are listed.

In Section~\ref{projectionbands0}
we deal with basic properties of band projections and projection bands in a pre-Riesz space $X$.
We show that every band projection is order continuous. 
If $X$ is a vector lattice and $B,C\sub X$ are ideals with $X=B\oplus C$, then $B$ is a projection band, see \cite[Theorem 1.41]{PosOp}. We give an example that this statement is not true in pre-Riesz spaces, even if $B$ and $C$ are bands. We give two different sets of conditions such that the statement is satisfied.

Section~\ref{projectionbands} is devoted to the restriction and extension of projection bands. 
If $X$ is pervasive, then a projection band in $X$ extends to a projection band in a vector lattice cover $Y$ of $X$; see Theorem \ref{extideal_extband_coincide}. 
In Theorem~\ref{projband_restrictable} we give conditions such that the restriction of a projection band in $Y$ is a projection band in $X$. This implication is not true, in general.

In Section~\ref{atoms} we introduce atoms and discrete elements in ordered vector spaces. We show that in a pre-Riesz space every atom is a discrete element and that the converse is not true, in general. In Theorem~\ref{atomic.4} we first establish that in an Archimedean pervasive pre-Riesz space $X$ atoms and discrete elements coincide. Moreover, atoms in $X$ correspond to atoms in $Y$.

In Section~\ref{fin_dim} we deal with finite-dimensional spaces. Using the theory of atoms, we characterize finite-dimensional Archimedean pervasive pre-Riesz spaces. In Theorem~\ref{atomic.101a} we show that these spaces are precisely the vector lattices.

In Section~\ref{principal_bands} we consider principal bands generated by atoms in an Archimedean pervasive pre-Riesz space $X$. We show that the ideal $\mathcal{I}_a$ generated by an atom $a\in X$ and the band $\B_a$ coincide. In Theorem~\ref{atomic.9} we generalize Theorem~\ref{LuxZahn} and show that every Archimedean pervasive pre-Riesz space admits the decomposition $X=\B_a \oplus \B_a^{\t{d}}$, provided $a\in X$ is an atom. We conclude that -- similar to the vector lattice case -- there exists an order projection onto the band $\B_a$.

\section{Preliminaries}\label{preliminaries}
We list some basic terminology in partially ordered vector spaces.
Let $X$ be a real vector space and let $X_+$ be a \emph{cone} in $X$, that is,  $X_+$ is a wedge ($x,y\in X_+$ and $\lambda,\mu\geq 0$ imply $\lambda x + \mu y \in X_+$) and $X_+ \cap (-X_+) =\left\{0\right\}$. In $X$ a partial order is defined by $x\leq y$ whenever $y-x\in X_+$. 
The space $(X,X_+)$ (or, loosely, $X$) is then called a (\emph{partially}) \emph{ordered vector space}.  For a linear subspace $D$ of $X$ we consider in $D$ the order induced from $X$, i.e.\ we set $D_+:=D\cap X_+$. A non-empty convex subset $F$ of a cone $X_+$ is called a \emph{face} if $\alpha x + (1-\alpha)y \in F$ with $x,y\in X_+$ and $0<\alpha<1$ imply $x,y\in F$. For a subset $A$ of $X$ and $\lambda\in\RR$ we use the notations $A+x:=\left\{a+x\mid| a\in A\right\}$ and $\lambda A :=\left\{\lambda a \mid| a\in A\right\}$. The \emph{positive-linear hull} of $A\sub X$ is given by
\[\textnormal{pos}A:=\left\{ x\in X\mid| \exists n\in\NN, \lambda_i\in\RR_{>0}, x_i\in A, i=1,\ldots, n \t{ with } x=\sum_{i=1}^n \lambda_i x_i\right\}.\]

An ordered vector space $X$ is called \emph{Archimedean} if for every $x,y\in X$ with $nx\leq y$ for every $n\in\NN$ one has $x\leq 0$. Clearly, every subspace of an Archimedean ordered vector space is Archimedean.
A subspace $D\sub X$ is called \emph{directed} if for every $x,y\in D$ there is an element $z\in D$ such that $x\leq z$ and $y\leq z$. An ordered vector space $X$ is directed if and only if  $X_+$ is \emph{generating} in $X$, that is, $X = X_+ - X_+$. The ordered vector space $X$ has the \emph{Riesz decomposition property} (\emph{RDP}) if for every $x_1,x_2,z \in X_+$ with $z \leq x_1+x_2$ there exist $z_1,z_2 \in X_+$ such that $z = z_1+z_2$ with $z_1 \leq x_1$ and $z_2\leq x_2$. The space $X$ has the RDP if and only if for every $x_1,x_2,x_3,x_4\in X$ with $x_1,x_2 \leq x_3,x_4$ there exists $z\in X$ such that $x_1, x_2 \leq z \leq x_3, x_4$. An element $u\in X_+$ is called an \emph{order unit}, if for every $x\in X$ there is $\lambda\in\RR_{>0}$ with $-\lambda u \leq x \leq \lambda u$.
A net $(z_\alpha)_{\alpha\in\mathbb{A}}$ in $X$  is \emph{decreasing}, in symbols $z_\alpha\downarrow$, if for every $\alpha,\beta\in\mathbb{A}$ with $\alpha\leq\beta$ we have $z_\alpha\geq z_\beta$. We write $z_\alpha\downarrow 0$ if $z_\alpha\downarrow$ and $\inf\left\{z_\alpha\mid|\alpha\in\mathbb{A}\right\}=0$.
A net $(x_\alpha)_\alpha$ in $X$ \emph{order converges} to $x\in X$, in symbols $x_\alpha\xrightarrow{o}x$, if there is a net $(z_\alpha)_\alpha$ in $X$ with $z_{\alpha}\downarrow 0$ and $\pm(x-x_\alpha)\leq z_\alpha$ for every $\alpha$.

For an ordered vector space $(X,X_+)$, a linear operator $T\colon X\rightarrow X$ is called \emph{positive}, if $T(X_+)\sub X_+$. The operator $T$ is positive if and only if for every $x,z\in X$ the relation $x\leq z$ implies $T(x)\leq T(z)$. If $T$ is positive, we write $T\geq 0$. If $X$ is directed, then this introduces a partial order on the space $L(X)$ of linear operators on $X$. An operator $T\in L(X)$ is \emph{order continuous}\footnote{There are several alternative definitions of order convergence of nets in partially ordered vector spaces. However, by \c[Theorem~4.4]{10} an operator $T$ is order continuous if and only if $T$ is continuous with respect to any of these alternative notions.} if for every net $(x_\alpha)_\alpha$ in $X$ with $x_\alpha\xrightarrow{o}x\in X$ we have $T(x_\alpha)\xrightarrow{o}T(x)$.
We recall the following characterization from \c[Lemma~7]{04}. 
\begin{lemma}\label{charOCts}
Let $X$ and $Z$ be ordered vector spaces, $Z$ Archimedean, and $T\colon X\rightarrow Z$ a positive operator. Then the following are equivalent.
\begin{enumerate}
\itemsep0em
\item $T$ is order continuous.
\item For every net $(x_\alpha)_\alpha$ in $X$ with $x_\alpha\downarrow 0$ it follows $T(x_\alpha)\downarrow 0$.
\end{enumerate}
\end{lemma}

For standard notations in the case that $X$ is a vector lattice see \c{PosOp}. We will use the following simple observation.
\begin{lemma}\label{prelim.0}
Let $X$ be a vector lattice.
\begin{enumerate}
\itemsep0em
\item\label{prelim.0.it1} Let $a,b,c\in X_+$ and $a\perp b$. Then $(a+b)\Inf c = a\Inf c + b\Inf c$.
\item\label{prelim.0.it2} Let $x\geq 0$ be such that $x=x_1+x_2$ with $x_1\perp x_2$. Then $x_1,x_2\geq 0$.
\end{enumerate}
\end{lemma}
\begin{proof}
\ref{prelim.0.it1}: Due to $a\Inf b=0$, by  \c[Theorem~1.3(2)]{PosOp} we have $a+b  = a\Sup b$. Similarly, $(a\Inf c)\perp (b\Inf c)$ implies $(a\Inf c) + (b\Inf c) = (a\Inf c)\Sup(b\Inf c)$. Thus, using \c[Proposition~1.1.2]{MeyNie} in the second step of the following equation, we obtain
\[(a+b)\Inf c = (a\Sup b)\Inf c = (a\Inf c)\Sup (b\Inf c) = (a\Inf c)+ (b\Inf c).\]
\ref{prelim.0.it2}: Due to $|x_1|\Inf|x_2|=0$, by \c[Theorem~1.7(5)]{PosOp} we have $|x_1+x_2|=|x_1-x_2|$. Since $x$ is positive, we get $x=x_1+x_2=|x_1+x_2|$. By \c[Theorem~1.7(7)]{PosOp} we obtain $x=|x_1+x_2|\Sup |x_1-x_2| = |x_1|+|x_2|$. It follows $x_1+x_2 =x= |x_1|+|x_2|$ and therefore $x_2=(|x_1|-x_1)+|x_2|\geq 0$. Analogously, $x_1\geq 0$.
\end{proof}
The next proposition is from \c[Proposition~III.10.1~b)]{Vulikh_en}.
\begin{proposition}\label{prelim.17bb}
Let $Y$ be an Archimedean vector lattice, $y\in Y_+$ and $\Lambda\sub\RR$ a bounded set. Then we have 
$\sup \left\{\lambda y\mid| \lambda\in\Lambda \right\} = (\sup \Lambda) y$ and $\inf\left\{\lambda y\mid| \lambda\in\Lambda \right\} = (\inf \Lambda) y$.
\end{proposition}

Finite dimensional vector lattices are characterized as follows, see \c[Theorem 1.7.8]{PRS}.
\begin{proposition}\label{Yudin}
Let $(X,K)$ be an $n$-dimensional Archimedean directed ordered vector space. Then the following statements are equivalent.
\begin{enumerate}
\itemsep0em
\item\label{Yudin.it1} $(X,K)$ is a vector lattice.
\item\label{Yudin.it2} $(X,K)$ is order isomorphic to $(\RR^n,\RR^n_+)$.
\item\label{Yudin.it3} $(X,K)$ has the RDP.
\end{enumerate}
\end{proposition}
Recall that a vector lattice $X$ is \emph{Dedekind complete} if every non-empty subset of $X$ which is bounded above has a supremum.

We say that a linear subspace $D$ of a vector lattice $X$ \emph{generates $X$ as a vector lattice} if for every $x\in X$ there exist $a_1, \ldots, a_n, b_1, \ldots, b_m \in D$ such that $x = \bigvee_{i=1}^{n}a_i - \bigvee_{j=1}^{m}b_j$. 
We call a linear subspace $D$ of an ordered vector space $X$ \emph{order dense} in $X$ if for every $x\in X$ we have \[x = \inf\left\{z\in D \mid| x\leq z\right\},\] that is, the greatest lower bound of the set $\left\{z\in D \mid| x\leq z\right\}$ exists in $X$ and equals $x$, see \c[p.~360]{113}.
Recall that a linear map $i\colon X\rightarrow Y$, where $X$ and $Y$ are ordered vector spaces, is called \emph{bipositive} if for every $x\in X$ one has $i(x) \geq 0$ if and only if $x\geq 0$. An \emph{embedding} is a bipositive linear map, which implies injectivity. 
For an ordered vector space $X$, the following statements are equivalent, see \c[Corollaries~4.9-11 and Theorems~3.7, 4.13]{vanHaa}:
\begin{enumerate}
	\itemsep0em
	\item\label{embedding.it1} There exist a vector lattice $Y$ and an embedding $i\colon X\rightarrow Y$ such that $i(X)$ is order dense in $Y$.
	\item\label{embedding.it2} There exist a vector lattice $\tilde{Y}$ and an embedding $i\colon X\rightarrow \tilde{Y}$ such that $i(X)$ is order dense in $\tilde{Y}$ and generates $\tilde{Y}$ as a vector lattice.
\end{enumerate}
If $X$ satisfies \ref{embedding.it1}, then $X$ is called a \emph{pre-Riesz space}, and $(Y,i)$ is called a \emph{vector lattice cover} of $X$. For an intrinsic definition of pre-Riesz spaces see \c{vanHaa}. If $X$ is a subspace of $Y$ and $i$ is the inclusion map, we write briefly $Y$ for $(Y,i)$.
As all spaces $\tilde{Y}$ in \ref{embedding.it2} are Riesz isomorphic, we call the pair $(\tilde{Y},i)$ the \emph{Riesz completion of} $X$ and denote it by $X^\rho$. 
The space $X^\rho$ is the smallest vector lattice cover of $X$ in the sense that every vector lattice cover $Y$ of $X$ contains a Riesz subspace that is Riesz isomorphic to $X^\rho$.
For the following result see  \c[Corollary~5]{02} or \c[Proposition 1.6.2]{PRS}.
\begin{proposition}\label{properties.23y}
Let $X$ be a pre-Riesz space, $(Y,i)$ a vector lattice cover of $X$ and $S\sub X$ a non-empty subset.
\begin{enumerate}
\itemsep0em
\item\label{properties.23y.it1} If $\sup S$ exists in $X$, then $\sup i(S)$ exists in $Y$ and $\sup i(S) = i(\sup S)$.
\item\label{properties.23y.it2} If $\sup i(S)$ exists in $Y$ and $\sup i(S)\in i(X)$, then $\sup S$ exists in $X$ and $\sup i(S) = i(\sup S)$.
\end{enumerate}
\end{proposition}

By \c[Theorem 17.1]{vanHaa} every Archimedean directed ordered vector space is a pre-Riesz space. Moreover, every pre-Riesz space is directed.
If $X$ is an Archimedean directed ordered vector space, then every vector lattice cover of $X$ is Archimedean.  
By \c[Chapter~X.3]{VulikhWeber} every Archimedean directed  ordered vector space $X$ has a unique (up to isomorphism) Dedekind completion, which we denote by $X^\delta$. Clearly, $X^\delta$ is a vector lattice cover of $X$.

Let $X$ be a pre-Riesz space and $(Y,i)$ a vector lattice cover of $X$. By \cite[Theorem~4.14]{vanHaa} and \cite[Theorem~4.5]{vanHaa} the Dedekind completions $X^\delta$ and $Y^\delta$ are order and linearly isomorphic, i.e.\ we can identify $X^\delta = Y^\delta$.
A subspace $D$ of $Y$ is called \emph{pervasive in} $Y$, if for every $y\in Y_+$, $y\neq 0$, there exists $d\in D$ such that $0<d\leq y$. 
If $D$ is a pervasive and majorizing subspace of $Y$, then by \c[Proposition~2.8.5]{PRS} $D$ is order dense in $Y$.
By \c[Proposition~2.8.8]{PRS} the pre-Riesz space $X$ is pervasive in $Y$ if and only if $X$ is pervasive in any vector lattice cover. Then $X$ is simply called \emph{pervasive}.
In the following result we give characterizations of pervasive pre-Riesz spaces.
The characterization in \ref{properties.1.it3} is from \cite[Theorem~4.15 and Corollary~4.16]{Waaij} (a short proof can also be found in \c[Lemma~1]{02}) and the one in \ref{properties.1.it4} from \c[Proposition~6]{05}.
\begin{proposition}\label{properties.1}
Let $X$ be an Archimedean pre-Riesz space and $(Y,i)$ a vector lattice cover of $X$. Then the following statements are equivalent.
\begin{enumerate}
\itemsep0em
\item\label{properties.1.it1} $X$ is pervasive.
\item\label{properties.1.it3} For every $y\in Y_+$ with $y\neq 0$ we have $y = \sup\left\{x\in i(X)\mid| 0\leq x\leq y\right\}$.
\item\label{properties.1.it4} For every $b\in X$ with $i(b)\Sup 0> 0$ there exists $x\in X$ such that $0 < i(x) \leq i(b)\Sup 0$.
\end{enumerate}
\end{proposition}
Clearly, in the formula in (ii)  we can replace the set $\left\{x\in i(X) \mid| 0\le x\leq y\right\}$ by the set $\left\{x\in i(X) \mid| 0 < x\leq y\right\}$.
The pre-Riesz space $X$ is called \emph{fordable in} $Y$ if for every $y\in Y$ there exists a set $S\sub X$ such that $\left\{y\right\}^{\t{d}} = i(S)^{\t{d}}$ in $Y$. By \c[Proposition~4.1.18]{PRS} the space $X$ is fordable in $Y$ if and only if $X$ is fordable in any vector lattice cover of $X$. Then $X$ is simply called \emph{fordable}. By \c[Proposition~4.1.15]{PRS} every pervasive pre-Riesz space is fordable.

Next we define disjointness, bands and ideals in ordered vector spaces.
For a subset $M\sub X$ denote the set of upper bounds of $M$ by $M^u:=\left\{x\in X\mid| \forall m\in M\colon m\leq x\right\}$. 
The elements $x,y\in X$ are called \emph{disjoint}, in symbols $x\perp y$, if 
\begin{equation}\label{disjoint}
\left\{x+y,-x-y\right\}^u = \left\{x-y,-x+y\right\}^u,
\end{equation}
for motivation and details see \c{1}. If $X$ is a vector lattice, then this notion of disjointness coincides with the usual one, see \c[Theorem~1.4(4)]{PosOp}. 
The following result is established in \c[Proposition~2.1(ii)]{1}.
\begin{proposition}\label{prelim.36}
Let $Y$ be a pre-Riesz space and $(Y,i)$ a vector lattice cover of $X$. Then $x\perp y$ holds in $X$ if and only if $i(x)\perp i(y)$ holds in $Y$.
\end{proposition}
The subsequent result is obtained in \c[Theorem~22]{03}.
\begin{theorem}\label{closedness.7supperp}
Let $X$ be an Archimedean pervasive pre-Riesz space, $a\in X$ and $S\sub X_+$ such that $\sup S$ exists in $X$. Then the relation $a\perp S$ implies $a\perp\sup S$.
\end{theorem}
Let $X$ be an ordered vector space. The \emph{disjoint complement} of a set $M\sub X$ is $M^{\t{d}} := \left\{x\in X \mid| \forall y\in M\colon x\perp y\right\}$. A linear subspace $B$ of $X$ is called a \emph{band} in $X$ if $B\dd = B$, see \c[Definition~5.4]{1}. If $X$ is an Archimedean vector lattice, then this notion of a band coincides with the classical notion of a band in vector lattices (where a band is defined to be an order closed ideal). For every subset $M\sub X$ the disjoint complement $M^{\t{d}}$ is a band, see \c[Proposition~5.5]{1}. 
For an element $a\in X$, by $\B_a$ we denote the \emph{band generated by} $a$, i.e.\ $\B_a:= \bigcap \left\{B\sub X\mid| B \t{ is a band in } X \t{ with } a\in B\right\}$. The band $\B_a$ is called a \emph{principal band}. By \c[Lemma~4]{03} we have $\B_a =\left\{a\right\}\dd$.

The following notion of an ideal is introduced in \c[Definition~3.1]{vanGaa}. A subset $M$ of an ordered vector space $X$ is called \emph{solid} if for every $x\in X$ and $y\in M$ the relation
$\left\{x,-x\right\}^u \supseteq \left\{y,-y\right\}^u$ implies $x\in M$.
A solid subspace of $X$ is called an \emph{ideal}. This notion of an ideal coincides with the classical definition, provided $X$ is a vector lattice. For an element $a\in X$, by $\I_a$ we denote the \emph{ideal generated by} $a$, i.e.\ $\I_a:= \bigcap \left\{I\sub X\mid| I \t{ is an ideal in } X \t{ with } a\in I\right\}$, see also \c[Lemma~5]{03}. The ideal $\I_a$ is called a \emph{principal ideal}.

Next we discuss the restriction property and the extension property for ideals and bands.
Let $X$ be a pre-Riesz space and $(Y,i)$ a vector lattice cover of $X$. 
For $S\sub Y$ we write  $[S]i:=\left\{x\in X\mid| i(x)\in S\right\}$. 
The pair $(L,M)\subseteq \mathcal{P}(X)\times \mathcal{P}(Y)$
is said to satisfy  
\begin{itemize}
	\item[-]
	the \emph{restriction property} (R), if whenever $J\in M$,
	then $[J]i\in L$, and
	\item[-]
	the \emph{extension property} (E), if whenever $I\in L$,
	then there is $J\in M$ such that 
	$I=[J]i$.
\end{itemize}
In \c{1} the properties (R) and (E) are investigated for ideals and bands.
It is shown that the extension property (E) is satisfied for  bands, i.e.\ for $L$ being the set of bands in $X$ and $M$ being the set of bands in $Y$.
Moreover, the restriction property (R) is satisfied for ideals. In general, bands do not have (R) and ideals do not have (E).  
The appropriate sets $M$ and $L$ of directed ideals satisfy (E). If $X$ is fordable, then by \c[Proposition~2.5 and Theorem~2.6]{3} we have (R) for bands.
If for an ideal $I$ in $X$ and an ideal $J$ in $Y$ we have $I=[J]i$, then $J$ is called an \emph{extension ideal} of $I$. An \emph{extension band} $J$ for a band $I$ in $X$ is defined similarly. Extension ideals and bands are not unique, in general.
If an ideal $I$ in $X$ has an extension ideal in $Y$, then $\hat{I}:=\bigcap\left\{J\subseteq Y \mid| J \t{ is an extension ideal of } I\right\}$ is the smallest extension ideal of $I$ in $Y$.
For a band $B$ in $X$ the smallest extension band $\hat{B}$ of $B$ is defined similarly.
From \c[Theorem~16]{03} we obtain the following.
\begin{lemma}\label{majorizing}
Let $X$ be a pre-Riesz space and $I$ an ideal in $X$. If $I$ is directed, then $i(I)$ is majorizing in $\hat{I}$. 
\end{lemma}

By \c[Proposition 17 (a)]{6} and its subsequent discussion the smallest extension band of $B$ is given by $\hat{B}=i(B)\dd$ in $Y$.
The following result can be found in \c[Theorem~31]{03}.
\begin{theorem}\label{closedness.1002}
Let $X$ be an Archimedean pervasive pre-Riesz space with a vector lattice cover $(Y,i)$. Let $B\sub X$ be a band and $D\sub Y$ an extension band of $B$. Then the following statements are equivalent.
\begin{enumerate}
\itemsep0em
\item\label{closedness.1002.it1} $i(B)$ is majorizing in $D$.
\item\label{closedness.1002.it2} $i(B)$ is order dense in $D$.
\end{enumerate}
If \ref{closedness.1002.it1} or \ref{closedness.1002.it2} are satisfied, then $B$ is a pre-Riesz space and hence directed and $D$ is a vector lattice cover of $B$. Moreover, $B$ is pervasive.
\end{theorem}

In the literature on vector lattices, atoms and discrete elements are defined in several  different ways. We use the following notions.
In a vector lattice $Y$, an element $a\in Y\ohne\{0\}$ is called an \emph{atom} if for every $y\in Y$ with $|y| \leq |a|$ there is $\lambda\in\RR$ such that $y=\lambda a$.
An element $d\in Y\ohne\{0\}$ is called \emph{discrete} if for every pair of disjoint elements $u,v\in Y$ with $0\leq u\leq |d|$ and $0\leq v\leq |d|$ it follows that $u=0$ or $v=0$, see \c[Definition~III.13.1]{Vulikh_en}.
The following statements can be found in \c[Chapter 2, Exercises~5.(i)]{PosOp}, \c[III.13.1~b)]{Vulikh_en}, \c[Lemma~26.2 (ii)]{Zaa1}.
\begin{proposition}\label{prelim.19.2}
Let $Y$ be an Archimedean vector lattice and $a\in Y_+\setminus\{0\}$. Then the following statements are equivalent.
\begin{enumerate}
\itemsep0em
\item\label{prelim.19.2.eq1} The element $a$ is an atom.
\item\label{prelim.19.2.eq2} The principal ideal $\I_a$ is one-dimensional.
\item\label{prelim.19.2.eq3} The element $a$ is a discrete element.
\end{enumerate}
\end{proposition}
Notice that if $Y$ is not Archimedean, then the implication from \ref{prelim.19.2.eq3} to \ref{prelim.19.2.eq1} is not true, in general.
For further results see \c[III.13.1 a) to d)]{Vulikh_en}.
The following observation is established in \c[Proposition~III.13.1~d)]{Vulikh_en}.
\begin{proposition}\label{prelim.19.6}
Let $Y$ be an Archimedean vector lattice, $a \in Y_+$ an atom and $x\in Y$ with $0<a\leq x$. Then there exists a real number $\lambda>0$ such that $x-\lambda a\geq 0$ and $(x-\lambda a)\perp a$.
\end{proposition}

\section{Projection bands in pre-Riesz spaces}\label{projectionbands0}
For a projection $P$ on an ordered  vector space $X$, i.e.\ for a linear operator  with $P^2=P$, there are two natural ways to relate $P$ with the order structure.
On one hand, if $A\sub X$ is a linear subspace, a projection $P\colon X\rightarrow X$ with $0\leq P \leq I$ and $P(X)=A$ is called an \emph{order projection} onto $A$. On the other hand, if $B\sub X$ is a band such that $X=B\oplus B^{\t{d}}$, then $B$ is called a \emph{projection band}, and the operator $P_B\colon X\rightarrow X$, $x\mapsto x_1$ is well-defined, where 
for every $x\in X$ we have the unique decomposition $x=x_1+x_2$ for $x_1\in B$ and $x_2\in B^{\t{d}}$. The operator $P_B$ is called the \emph{band projection} onto $B$.
Order projections are introduced in \c{GlueWolf19} and considered in \c{14} in relation to band projections. 
In the following statement we collect the results from \c[Theorem 3.2 and Propositions~2.5, 2.3 and 3.1]{14}.
\begin{theorem}\label{Glueck}
	Let $X$ be a pre-Riesz space.
	\begin{enumerate}
		\itemsep0em 
		\item\label{Glueck.it1} An operator $P\colon X\rightarrow X$ is an order projection if and only if it is a band projection.
		\item\label{Glueck.it2} Every projection band $B\sub X$ is directed.
		\item\label{Glueck.it3} A band $B\sub X$ is a projection band if and only if $B^{\t{d}}$ is a projection band.
		\item\label{Glueck.it4} If two band projections $P$ and $Q$ have the same range, then $P=Q$.
	\end{enumerate}
\end{theorem}


Order projections in vector lattices are order continuous. To observe the same fact in Archimedean pre-Riesz spaces, we need the following result. 
\begin{proposition}\label{o_continuity_preserved_by_domination}
	Let $X$ be an Archimedean pre-Riesz space, $S,T\colon X\rightarrow X$ two operators with $0\leq S\leq T$ and $T$ order continuous.
	Then $S$ is order continuous.
\end{proposition}
\begin{proof}
	As $P\geq 0$ and $X$ is Archimedean, by Lemma~\ref{charOCts} it suffices to show that for every net $(x_\alpha)_\alpha$ in $X_+$ with $x_\alpha \downarrow 0$ it follows $S(x_\alpha)\downarrow 0$. Let $(x_\alpha)_\alpha$ be a net in $X_+$ such that $x_\alpha \downarrow 0$. From $0\leq S\leq T$ for every $\alpha$ it follows
	\[0\leq S(x_\alpha)\leq T(x_\alpha).\]
	As $T$ is order continuous, from $x_\alpha \downarrow 0$ it follows $T(x_\alpha)\downarrow 0$. In particular, we have $\inf \left\{T(x_\alpha)\mid|\alpha\in\mathcal{A}\right\}=0$. Due to $S\geq 0$ we obtain $S(x_\alpha) \downarrow$. It is left to show that $\inf \left\{S(x_\alpha)\mid|\alpha\in\mathcal{A}\right\}=0$.
	If $v$ is a lower bound of the set $M:=\left\{S(x_\alpha)\mid|\alpha\in\mathcal{A}\right\}$, then $v$ is also a lower bound of the set  $\left\{T(x_\alpha)\mid|\alpha\in\mathcal{A}\right\}$, hence $v\leq 0$. As $0$ is a lower bound of $M$, we obtain $0=\inf M$.
\end{proof}

Since for every order projection $P$ we have $0\leq P\leq I$, where $I$ is the identity operator, from Proposition~\ref{o_continuity_preserved_by_domination} we immediately obtain the following.
\begin{corollary}\label{band_projection}
Let $X$ be an Archimedean pre-Riesz space and $P$ an order projection. 
Then $P$ is order continuous.
\end{corollary}

Recall that for two ideals $B,C$ in a vector lattice $X$ the relation $X=B\oplus C$ implies that $B$ is a projection band and $C=B^{\t{d}}$. The two subsequent Theorems~\ref{bandsareprojectbands} and \ref{directedbandsareprojectbands} generalize this result to pre-Riesz spaces.
\begin{theorem}\label{bandsareprojectbands}
	Let $X$ be a pervasive pre-Riesz space and $B,D\sub X$ two ideals such  
	that $X=B\oplus D$. Then $B$ is a projection band and $X=B\oplus  
	B^{\t{d}}$.
\end{theorem}
\begin{proof}
	Let $B,D\sub X$ be two ideals with $X=B\oplus D$. Assume, on the  
	contrary, that $B\not\perp D$. Then there exist two elements $b\in B$  
	and $d\in D$ such that $b\not\perp d$. For a vector lattice cover  
	$(Y,i)$ of $X$, by Proposition~\ref{prelim.36} it follows $i(b)\not\perp i(d)$. We obtain  
	$0<|i(b)|\Inf|i(d)|$. Since $X$ is pervasive, there exists $x\in X$  
	with $0<i(x)\leq |i(b)|\Inf|i(d)|$. We show $x\in B$. Indeed,  
	$i(x)\leq |i(b)|$ is equivalent to $\left\{i(x)\right\}^u\supseteq  
	\left\{i(b),-i(b)\right\}^u$. Taking the intersection of these sets of  
	upper bounds with $X$ we obtain due to $x>0$ that  
	$\left\{x,-x\right\}^u=\left\{x\right\}^u\supseteq  
	\left\{b,-b\right\}^u$ in $X$. It follows $x\in B$. Similarly, we  
	obtain $x\in D$. From the uniqueness of the decomposition $X=B\oplus  
	D$ we have $B\cap D=\left\{0\right\}$, which implies $x=0$, a  
	contradiction. We thus obtain $B\perp D$. It follows $D\sub B^{\t{d}}$.
	
	To establish $B^{\t{d}}\sub D$, let $x\in B^{\t{d}}$. From $X=B\oplus  
	D$ it follows $x=b+d$ with $b\in B$ and $d\in D\sub B^{\t{d}}$. Due to  
	$x,d\in B^{\t{d}}$ we obtain $b=x-d\in B^{\t{d}}$. This yields $b\in  
	B\cap B^{\t{d}}=\left\{0\right\}$ and therefore $x=d\in D$. This  
	establishes $B^{\t{d}}\sub D$. We conclude $B^{\t{d}}=D$.
\end{proof}

We show a result similar to Theorem~\ref{bandsareprojectbands}, where the ideals are directed and the condition on the underlying space is weaker. A pre-Riesz space $X$ is called \emph{weakly pervasive} if for every $b,d\in X$ with $b,d>0$ and $b\not\perp d$ there exists $x\in X$ such that $0<x\leq b,d$, see \c[Definition~8 and Lemma~9]{05}. There it is also shown that every pervasive pre-Riesz space is weakly pervasive.

\begin{theorem}\label{directedbandsareprojectbands}
Let $X$ be a weakly pervasive pre-Riesz space and $B,D\sub X$ two directed ideals such that $X=B\oplus D$. Then $B$ is a projection band and $X=B\oplus B^{\t{d}}$.
\end{theorem}
\begin{proof}
Let first $b\in B_+$. We show for every $d\in D_+$ that $b\perp d$. Let $d\in D_+$. Assume, on the contrary, that $b\not\perp d$. Since $X$ is weakly pervasive, there exists $\tilde{x}\in X_+$ with $0< \tilde{x}\leq b,d$. Due to $B,D$ being ideals we obtain $\tilde{x}\in B,D$. From the uniqueness of the decomposition $X=B\oplus D$ we have $B\cap D=\left\{0\right\}$, which is a contradiction to $\tilde{x}>0$. It follows $b\perp d$. Since $X$ is a pre-Riesz space, disjoint complements are linear subspaces of $X$. 
Since $D$ is directed, it follows $b\perp D$. The directedness of $B$ implies $B\perp D$. We conclude $D\sub B^{\t{d}}$.
	
To establish $B^{\t{d}}\sub D$, let $x\in B^{\t{d}}$. From $X=B\oplus D$ it follows $x=b+d$ with $b\in B$ and $d\in D\sub B^{\t{d}}$. Due to $x,d\in B^{\t{d}}$ we obtain $b=x-d\in B^{\t{d}}$. This yields $b\in B\cap B^{\t{d}}=\left\{0\right\}$ and therefore $x=d\in D$. This establishes $B^{\t{d}}\sub D$. We conclude $B^{\t{d}}=D$.
\end{proof}

The following example demonstrates that, in general, Theorem~\ref{bandsareprojectbands} is not true if $X$ is not pervasive, and Theorem~\ref{directedbandsareprojectbands} is not true if one of the ideals is not directed.

\begin{example}
\textit{A weakly pervasive pre-Riesz space $X$ with two bands $B,C \sub X$ such that $X=B \oplus C$ and $C \not= B^{\t{d}}$.}
	
Let $\ell^\infty(\ZZ)$ be the vector space of bounded sequences on $\ZZ$ endowed with the pointwise order. In \c[Example 5.2]{2} it is shown that
$Y:=\left\{(y_k)_k \in \ell^\infty(\ZZ) \mid| \Lim{k\to\infty} y_k \t{ exists}\right\}$ is a vector lattice and its linear subspace
\[X:=\left\{(x_k)_k \in \ell^\infty(\ZZ) \mid| L:=\lim_{k\to\infty} x_k \t{ exists and } \sum_{k=1}^\infty \frac{x_{-k}}{2^k} = L\right\}\] is a pre-Riesz space and order dense in $Y$. Thus by Proposition~\ref{prelim.36} disjointness is pointwise in $X$. In \c[Example~10]{05} it is established that $X$ is weakly pervasive.
Consider the following subspaces of $X$:
\[B:=\left\{(x_k)_k \in X \mid| x_k=0 \t{ for } k\leq -2\right\} \hs\t{ and }\hs C:=\left\{(x_k)_k \in X \mid| x_k=0 \t{ for } k\geq 0\right\}.\]
We first show that $B$ is a directed band.
Let $a=(a_k)_k$, $b=(b_k)_k\in B$. For $k\leq -2$ we have $a_k=0$, which implies $\Lim{k\to\infty} a_k = \sum_{k=1}^\infty \frac{a_{-k}}{2^k} = \frac{1}{2}a_{-1}$. Similarly, $\frac{1}{2}b_{-1}=\Lim{k\to\infty} b_k$. Define a bounded sequence $c=(c_k)_k$ by $c_k:=\frac{1}{2}\max\left\{a_k,b_k\right\}$ for every $k\in\ZZ$. Then $\Lim{k\to\infty} c_k=\frac{1}{2}\max\left\{a_{-1}, b_{-1}\right\}=\frac{1}{2}c_{-1}$. Moreover, for $k\in \ZZ_{\leq -2}$ we have $c_k=0$. It follows $\sum_{k=1}^\infty \frac{c_{-k}}{2^k}=\frac{1}{2}c_{-1}=\Lim{k\to\infty} c_k$, i.e.\ $c\in X$. Clearly, $c\in B$ and $a,b\leq c$. Thus $B$ is directed. We show that $B$ is a band. For every $n\in\NN_{\geq 1}$ define a sequence $x^{(n)}=(x^{(n)}_k)_k\in X$ with $x^{(n)}_k:=0$ for $k\in\ZZ\ohne\left\{-n,-(n+1)\right\}$ and $x^{(n)}_{-n}:=1$, $x^{(n)}_{-(n+1)}:=-2$. Then for every $n\in\NN_{\geq 1}$ we have $\sum_{k=1}^\infty \frac{x^{(n)}_{-k}}{2^k}=\frac{x^{(n)}_{-n}}{2^n}+\frac{x^{(n)}_{-(n+1)}}{2^{n+1}}=0=\Lim{k\to\infty} x^{(n)}_k$. That is, $x^{(n)}\in X$. Moreover, for every $n\in\NN_{\geq 2}$ we have $x^{(n)}\perp B$. For the set $S:=\left\{x^{(n)}\mid|n\in\NN_{\geq 2}\right\}\sub X$ we then have $S^{\t{d}}=B$. It follows that $B$ is a band.
	
Next we establish that $C$ is a non-directed band. To show that $C$ is not directed, consider the two elements $x^{(1)}, x^{(2)}\in C$. Assume that there exists $c=(c_k)_k\in C$ with $x^{(1)}, x^{(2)}\leq c$. Then $c_{-1}, c_{-2}\geq 1$ and $c_{-3}\geq -2$. Moreover $c_k\geq 0$ for every $k\in\ZZ_{\leq -3}$ and due to $c\in C$ we have $c_k=0$ for $k\in\ZZ_{\geq 0}$. It follows
\[\sum_{k=1}^\infty \frac{c_{-k}}{2^k}\geq \frac{c_{-1}}{2^1}+\frac{c_{-2}}{2^2}+\frac{c_{-3}}{2^3}\geq \frac{1}{2}+\frac{1}{4}-\frac{2}{8}>0,\] 
from which we obtain  $\Lim{k\to\infty} c_k >0$, a contradiction to $c_k=0$ for $k\in\ZZ_{\geq 0}$. Thus $C$ is not directed. To see that $C$ is a band, for every $n\in\NN_{\geq 0}$ define a sequence $z^{(n)}=(z^{(n)}_k)_k$ by $z^{(n)}_n:=1$ and $z^{(n)}_k:=0$ for $k\neq n$. Clearly, for every $n\in\NN_{\geq 0}$ we have $z^{(n)}\in X$. For the set $T:=\left\{z^{(n)}\mid|n\in\NN_{\geq 0}\right\}\sub X$ we obtain $T^{\t{d}}=C$. It follows that $C$ is a band.
	
We show $B\not\perp C$. To that end, let $b=(b_k)_k$ be defined for $k\in\ZZ_{\leq -2}$ by $b_k:=0$, for $k\in\ZZ_{\geq 0}$ by $b_k=\frac{1}{2}$ and $b_{-1}:=1$. Then $b\in B$. In $Y$ we have $i(b)\not\perp i(x^{(1)})$. By Proposition~\ref{prelim.36} for the sequence $x^{(1)}\in C$  and $b\in C$ we have $b\not\perp x^{(1)}$. It follows $B\not\perp C$.
	
Next we show $X=B\oplus C$. Let $x=(x_k)_k\in X$. Define $b=(b_k)_k, c=(c_k)_k\in X$ by
\begin{align*}
& b_k:=0 \t{ for } k\in\ZZ_{\leq 2}, &&\quad\t{ and }\quad && c_k:=x_k \t{ for } k\in\ZZ_{\leq 2},\\
& b_k:=x_k \t{ for } k\in\ZZ_{\geq 0}, &&  &&   c_k:=0 \t{ for } k\in\ZZ_{\geq 0},\\
& b_{-1}:=2\lim_{k\to\infty}x_k &&  &&  c_{-1}:=-2\sum_{k=2}^\infty \frac{x_{-k}}{2^k}.
\end{align*}
Then $b\in B$ and $c\in C$. Clearly, for $k\in\ZZ\ohne\left\{-1\right\}$ we have $x_k=b_k+c_k$. Moreover, $x_{-1}=2\frac{x_{-1}}{2}=2\big(\Lim{k\to\infty}x_k -\sum_{k=2}^\infty \frac{x_{-k}}{2^k}\big) = b_{-1}+c_{-1}$. That is, $x=b+c$. We show that this decomposition is unique. Let $\tilde{b}=(\tilde{b}_k)_k\in B$ and $\tilde{c}=(\tilde{c}_k)_k\in C$ be such that $x=\tilde{b}+\tilde{c}$. From $b,\tilde{b}\in B$ and $c,\tilde{c}\in C$ it follows for $k\in\ZZ\ohne\left\{-1\right\}$ that $b_k=\tilde{b}_k$ and $c_k=\tilde{c}_k$. It is left to show $b_{-1}=\tilde{b}_{-1}$ and $c_{-1}=\tilde{c}_{-1}$. As $\tilde{b}\in B$, we have $\sum_{k=1}^\infty \frac{\tilde{b}_{-k}}{2^k}= \frac{\tilde{b}_{-1}}{2}$. Since for $k\in\ZZ_{\geq 0}$ we have $x_k=\tilde{b}_k+\tilde{c}_k=\tilde{b}_k$, we obtain
\[\frac{\tilde{b}_{-1}}{2} = \sum_{k=1}^\infty \frac{\tilde{b}_{-k}}{2^k} = \lim_{k\to\infty} \tilde{b}_k = \lim_{k\to\infty} x_k,\]
from which it follows $\tilde{b}_{-1} = 2\Lim{k\to\infty} x_k = b_{-1}$. Moreover, as $\tilde{c}\in C$, we have $0=\Lim{k\to\infty}\tilde{c}_k = \sum_{k=1}^\infty \frac{\tilde{c}_{-k}}{2^k}= \frac{\tilde{c}_{-1}}{2} + \sum_{k=2}^\infty \frac{\tilde{c}_{-k}}{2^k}$.
As $\tilde{c}_k= x_k$ for $k\in\ZZ_{\leq -2}$, we get $\tilde{c}_{-1} = -2\sum_{k=2}^\infty \frac{\tilde{c}_{-k}}{2^k} = -2\sum_{k=2}^\infty \frac{x_{-k}}{2^k}= c_{-1}$. We conclude $b=\tilde{b}$ and $c=\tilde{c}$.

Finally, we show that $X$ is not pervasive.
To that end, we use the characterization of pervasiveness in Proposition~\ref{properties.1}~\ref{properties.1.it4}. Consider the sequence $x^{(1)}\in X$ defined above. The sequence $i(x^{(1)})\Sup 0 \in Y$ is zero in every coordinate, except in $(i(x^{(1)})\Sup 0)_{-1}=1$. Thus $i(x^{(1)})\Sup 0>0$. Assume, on the contrary, that $X$ is pervasive, that is, there is $x\in X$ with $0< i(x)\leq i(x^{(1)})\Sup 0$. Then for every $k\in\ZZ\ohne\left\{-1\right\}$ we have $x_k=0$ and it follows $\Lim{k\to\infty} x_k =0=\sum_{k=1}^\infty \frac{x_{-k}}{2^k}=\frac{x_{-1}}{2}$. We obtain $x=0$, a contradiction. We conclude that $X$ is not pervasive. 
\end{example}

\section{Extension and restriction of projection bands}\label{projectionbands}
In view of the embedding technique for pre-Riesz spaces, the question arises how projection bands in a pre-Riesz space $X$ are related to  projection bands in a vector lattice cover $Y$ of $X$. The restriction and extension properties for bands in pervasive pre-Riesz spaces suggest that there projection bands are linked in a similar way. 
We show that indeed the extension property for projection bands is satisfied in Archimedean pervasive pre-Riesz spaces, whereas for the restriction property stronger assumptions are needed.  

The next statement is a technical result that we use to establish Theorem~\ref{extideal_extband_coincide}.
\begin{lemma}\label{idealinclusions}
Let $X$ be an Archimedean pervasive pre-Riesz space and $(Y,i)$ a vector lattice cover of $X$. Let $B\sub X$ be a band such that $X=B\oplus B^{\t{d}}$. Moreover, let $\hat{B}$ be the
smallest extension ideal of $B$ and $A$ the smallest extension ideal of $B^{\t{d}}$ in $Y$.
Then $A\sub \hat{B}^{\t{d}}$ and $\hat{B}\sub A^{\t{d}}$.
\end{lemma}
\begin{proof}
The bands $B$ and $B^{\t{d}}$ are ideals and by Theorem~\ref{Glueck}\ref{Glueck.it2} directed. Therefore in $Y$ there exist the smallest extension ideals $\hat{B}$ of $B$ and  $A$ of $B^{\t{d}}$. 

Let $y_1\in\hat{B}_+$. Due to $X$ being Archimedean and pervasive, by Proposition~\ref{properties.1} we have $y_1=\sup\left\{z\in i(X)\mid| 0\leq z\leq y_1\right\}$. Since $B$ is a directed ideal, by Lemma~\ref{majorizing} the subspace $i(B)$ is majorizing in $\hat{B}$. Thus there exists $b\in B$ with $0\leq y_1\leq i(b)$. For every $x\in X$ with $0\leq i(x) \leq y_1$, due to $0\leq x\leq b$ it follows $x\in B$. This yields
\begin{equation}\label{extideal_extband_coincide.eq1}
y_1=\sup\left\{z\in i(B)\mid| 0\leq z\leq y_1\right\}.
\end{equation}
Since $B$ is a band in $X$, for every $x\in B$ and $d\in B^{\t{d}}$ we have $x\perp d$, which implies $\left\{z\in i(B)\mid| 0\leq z\leq y_1\right\}\perp i(d)$.
Applying Theorem~\ref{closedness.7supperp} in the vector lattice $Y$, for every $d\in B^{\t{d}}$ we obtain $y_1\perp i(d)$. From the fact that $\hat{B}$ is directed we conclude
\begin{equation}\label{extideal_extband_coincide.eq2}
\forall y_1\in \hat{B} \hs\hs\forall d\in B^{\t{d}}\colon y_1\perp i(d).
\end{equation}
As the roles of $B$ and $B^{\t{d}}$ are interchangeable, similarly to \eqref{extideal_extband_coincide.eq1} for every $y_2\in A_+$ we obtain $y_2=\sup\left\{z\in i(B^{\t{d}})\mid| 0\leq z\leq y_2\right\}$. Due to \eqref{extideal_extband_coincide.eq2}, for every $y_1\in \hat{B}$ we have $y_1\perp \left\{z\in i(B^{\t{d}})\mid| 0\leq z\leq y_2\right\}$. By Theorem~\ref{closedness.7supperp} it follows $y_1\perp y_2$. Since $A$ is directed we conclude that for every $y_1\in \hat{B}$ and every $y_2\in A$ we have $y_1\perp y_2$.
It follows $A\sub \hat{B}^{\t{d}}$ and $\hat{B}\sub A^{\t{d}}$.
\end{proof}

\begin{theorem}\label{extideal_extband_coincide}
Let $X$ be an Archimedean pervasive pre-Riesz space and $B\sub X$ a band such that $X=B\oplus B^{\t{d}}$. Let $(Y,i)$ be a vector lattice cover of $X$ and $\hat{B}$ the smallest extension ideal of $B$ in $Y$. Then we have the following.
\begin{enumerate}
\itemsep0em
\item\label{extideal_extband_coincide.it1} $Y=\hat{B}\oplus\hat{B}^{\t{d}}$,
\item\label{extideal_extband_coincide.it2} $\hat{B}$ equals the smallest extension band of $B$,
\item\label{extideal_extband_coincide.it3a} $i(B)$ is majorizing in $\hat{B}$ and $i(B^{\t{d}})$ is majorizing in $\hat{B}^{\t{d}}$,
\item\label{extideal_extband_coincide.it3b} $\hat{B}^{\t{d}} = \widehat{B^{\t{d}}}$,
\item\label{extideal_extband_coincide.it3} $\hat{B}$ coincides with every extension band and every extension ideal of $B$,
\item\label{extideal_extband_coincide.it4} $\hat{B}$ and $\hat{B}^{\t{d}}$ are vector lattice covers of $B$ and $B^{\t{d}}$, respectively, and both $B$ and $B^{\t{d}}$ are pervasive.
\end{enumerate}
\end{theorem}
\begin{proof}
\ref{extideal_extband_coincide.it1}: To establish $Y=\hat{B}\oplus\hat{B}^{\t{d}}$, let first $y\in Y_+$. Since $i(X)$ is majorizing in $Y$, there exists $x\in X$ such that $0\leq y\leq i(x)$. Due to $X=B\oplus B^{\t{d}}$ there exist $x_1\in B$ and $x_2\in B^{\t{d}}$ such that $x=x_1+x_2$. The vector lattice $Y$ has the RDP, therefore $0\leq y\leq i(x) = i(x_1)+i(x_2)$ implies that there are $y_1, y_2\in Y_+$ with $0\leq y_1\leq i(x_1)$ and $0\leq y_2\leq i(x_2)$ such that $y=y_1+y_2$. From $i(x_1)\in i(B)\sub\hat{B}$ and $0\leq y_1\leq i(x_1)$ we obtain $y_1\in\hat{B}$. Let $A$ be as in Lemma~\ref{idealinclusions}. Similarly, due to $i(x_2)\in A$, with Lemma~\ref{idealinclusions} we have $y_2\in A\sub\hat{B}^{\t{d}}$. Since $\hat{B}\perp\hat{B}^{\t{d}}$, it follows that $y=y_1+y_2$ is a disjoint decomposition.

We show that this decomposition is unique. Let there be another disjoint decomposition $y=\tilde{y}_1+\tilde{y}_2$ with $\tilde{y}_1\in\hat{B}$ and $\tilde{y}_2\in\hat{B}^{\t{d}}$. Due to the disjointness of the decompositions, with Lemma~\ref{prelim.0}\ref{prelim.0.it2} we have $\tilde{y}_1,\tilde{y}_2\geq 0$. Lemma~\ref{prelim.0}\ref{prelim.0.it1} then yields
\begin{align*}
\tilde{y}_1 &=\tilde{y}_1\Inf\tilde{y}_1 + \tilde{y}_2\Inf\tilde{y}_1=(\tilde{y}_1+\tilde{y}_2)\Inf\tilde{y}_1 = y\Inf\tilde{y}_1 = \\
	&=(y_1+y_2)\Inf\tilde{y}_1 = y_1\Inf\tilde{y}_1 + y_2\Inf\tilde{y}_1 = y_1\Inf\tilde{y}_1,
\end{align*}
i.e.\ $\tilde{y}_1\leq y_1$. Similarly, $y_1\leq\tilde{y}_1$. We conclude $y_1=\tilde{y}_1$. Analogously, we obtain $y_2=\tilde{y}_2$, i.e. the disjoint decomposition is unique.

Let $y\in Y$ be arbitrary. We have $y=y^+-y^-$ with unique disjoint decompositions $y^+=y^+_1+y^+_2$ and $y^-=y^-_1+y^-_2$, where $y^+_1,y^-_1\in \hat{B}$ and $y^+_2,y^-_2\in \hat{B}^{\t{d}}$. Due to $y^+\perp y^-$ it follows that $y=(y^+_1-y^-_1)+(y^+_2-y^-_2)$ is a unique decomposition of $y$ with $(y^+_1-y^-_1)\in \hat{B}$ and $(y^+_2-y^-_2)\in \hat{B}^{\t{d}}$. We conclude $Y=\hat{B}\oplus\hat{B}^{\t{d}}$.

\ref{extideal_extband_coincide.it2}: First we establish that $\hat{B}$ is a band. We need to show $\hat{B}\dd\sub\hat{B}$. Let $y\in\hat{B}\dd$ and consider first the case $y\geq 0$. By \ref{extideal_extband_coincide.it1} there is a decomposition $y=y_1+y_2$ with $y_1\in\hat{B}$ and $y_2\in\hat{B}^{\t{d}}$. Lemma~\ref{prelim.0}\ref{prelim.0.it2} yields $y_1,y_2\geq 0$. Therefore, by Lemma~\ref{prelim.0}\ref{prelim.0.it1} and due to $y\in\hat{B}\dd$, we obtain
\[y=y\Inf(y_1+y_2) = y\Inf y_1 + y\Inf y_2 = y_1 \in\hat{B}.\]
For an arbitrary $y\in\hat{B}\dd$ with $y=y^+-y^-$ it similarly follows $y^+,-y^-\in\hat{B}$ and therefore $y\in\hat{B}$.
This implies that $\hat{B}$ is a band. Since $\hat{B}$ is the smallest extension ideal of $B$ and every extension band $D$ of $B$ is an ideal, it follows $\hat{B}\sub D$, i.e.\ $\hat{B}$ is the smallest extension band of $B$.

\ref{extideal_extband_coincide.it3a}: To establish that $i(B)$ is majorizing in $\hat{B}$ it suffices to consider positive elements. Let $y\in \hat{B}_+$. Since $i(X)$ is majorizing in $Y$ and $X=B\oplus B^{\t{d}}$, there exists an $x\in i(X)$ with $y\leq x=x_1+x_2$, where $x_1\in i(B)$ and $x_2\in i(B^{\t{d}})$. As $Y$ has the RDP, there exist $y_1,y_2\in Y$ with $y=y_1+y_2$ such that $0\leq y_1\leq x_1$ and $0\leq y_2\leq x_2$. From $0\leq y_1\leq x_1 \in i(B)\sub\hat{B}$ it follows $y_1\in\hat{B}$. Then $y_1\in i(B)\sub \hat{B}$. The uniqueness of the decomposition established in \ref{extideal_extband_coincide.it1} yields $y=y_1$ and $y_2=0$. From $0\leq y=y_1\leq x_1\in i(B)$ we conclude that $i(B)$ is majorizing in $\hat{B}$.

To establish that $i(B^{\t{d}})$ is majorizing in $\hat{B}^{\t{d}}$, let $y\in(\hat{B}^{\t{d}})_+$. Similarly to the previous case there exist $x_1,x_2\in i(X)$ and $y_1,y_2\in Y$ with $y=y_1+y_2$ such that $0\leq y_1\leq x_1\in i(B)$ and $0\leq y_2\leq x_2\in i(B^{\t{d}})$. The uniqueness of the decomposition leads to $y_1=0$ and $y=y_2 \leq x_2\in i(B^{\t{d}})$, i.e.\ $i(B^{\t{d}})$ is majorizing in $\hat{B}^{\t{d}}$.

\ref{extideal_extband_coincide.it3b}: 
To establish the inclusion $\widehat{B^{\t{d}}}\sub \hat{B}^{\t{d}}$, let $A$ be the smallest extension ideal of $B^{\t{d}}$ in $Y$. Applying \ref{extideal_extband_coincide.it2} to $B^{\t{d}}$ we obtain that the ideal $A$ equals the smallest extension band of $B^{\t{d}}$, i.e.\ $\widehat{B^{\t{d}}}=A$. Due to Lemma~\ref{idealinclusions} it follows $\widehat{B^{\t{d}}}=A \sub \hat{B}^{\t{d}}$. 

To prove the inclusion $\hat{B}^{\t{d}}\sub\widehat{B^{\t{d}}}$, we show that $\widehat{B^{\t{d}}}\ohne\hat{B}^{\t{d}}=\leer$. Let, on the contrary, $y\in\widehat{B^{\t{d}}}\ohne\hat{B}^{\t{d}}$ with $y\neq 0$. We can assume that $y>0$, otherwise consider one of the two elements $y^+,y^-$. As $y\notin \hat{B}^{\t{d}}$, we have $y\not\perp\hat{B}$. That is, there exists a positive element $z\in\hat{B}$ such that $z\Inf y >0$. Since $X$ is pervasive, there exists $x\in X_+$ with $0<i(x)\leq z\Inf y$. The relation $0<i(x) \leq y\in \widehat{B^{\t{d}}}$ implies $i(x)\in \widehat{B^{\t{d}}}$ and thus $x\in [\widehat{B^{\t{d}}}]i=B^{\t{d}}$. On the other hand, the relation $0<i(x)\leq z\in \hat{B}$ implies $i(x)\in \hat{B}$ and thus $x\in [\hat{B}]i=B$, a contradiction. We conclude $\hat{B}^{\t{d}}\sub\widehat{B^{\t{d}}}$.

\ref{extideal_extband_coincide.it3}:
First we show that every extension ideal of $B$ is contained in $\hat{B}$.
Let $I\sub Y$ be an extension ideal of $B$. Let $y\in I_+$. From \ref{extideal_extband_coincide.it1} it follows $y=y_1+y_2$ for $y_1\in\hat{B}$ and $y_2\in\widehat{B^{\t{d}}}$. Due to $X$ being pervasive, by Proposition~\ref{properties.1} we have $y_2=\sup\left\{x\in i(X)\mid| 0\leq x\leq y_2\right\} = \sup i(\left[[0,y_2]\right]i)$. From $y_2\in\widehat{B^{\t{d}}}$ it follows $[0,y_2]\sub \widehat{B^{\t{d}}}$ and thus we obtain the first inclusion $\left[[0,y_2]\right]i\sub B^{\t{d}}$. Moreover, due to $y_1\perp y_2$, Lemma~\ref{prelim.0}~\ref{prelim.0.it2} yields $y_1,y_2\geq 0$. From $0\leq y_2\leq y\in I$ we obtain $[0,y_2]\in I$. Since $I$ is an extension ideal of $B$, we obtain the second inclusion $\left[[0,y_2]\right]i \sub [I]i = B$. The two inclusions together yield $\left[[0,y_2]\right]i\sub B^{\t{d}}\cap B =\left\{0\right\}$. This leads to $y_2= \sup i(\left[[0,y_2]\right]i)=0$. We conclude that $y=y_1\in\hat{B}$, i.e.\ $I\sub \hat{B}$.

Since for every extension ideal $I$ of $B$ we have $I\sub \hat{B}$ and due to $\hat{B}$ being the smallest extension ideal of $B$, it follows $I=\hat{B}$. That is, $\hat{B}$ coincides with every extension ideal of $B$.

We show that $\hat{B}$ equals every extension band of $B$. Let $C$ be an extension band of $B$. In particular, $C$ is an extension ideal of $B$, which implies $C\sub\hat{B}$. On the other hand, by \ref{extideal_extband_coincide.it2} the ideal $\hat{B}$ is the smallest extension band of $B$, so $\hat{B}\sub C$. That is, $\hat{B}=C$.

\ref{extideal_extband_coincide.it4}: As was established in \ref{extideal_extband_coincide.it2}, due to $B$ being a directed ideal, the subspace $i(B)$ is majorizing in the band $\hat{B}$. Thus by Theorem~\ref{closedness.1002} the band $\hat{B}$ is a vector lattice cover of $B$ and $B$ is pervasive. As the roles of $B$ and $B^{\t{d}}$ are interchangeable, a similar statement follows for the band $B^{\t{d}}$ and its extension band $\hat{B}^{\t{d}}$.
\end{proof}
Note that the assumptions of Theorem~\ref{extideal_extband_coincide} do not imply that $X$ is a vector lattice, see Example~\ref{atomic.6} and Theorem~\ref{atomic.9} below.

From Theorem~\ref{extideal_extband_coincide}\ref{extideal_extband_coincide.it1} and \ref{extideal_extband_coincide.it2} we obtain the following.
\begin{corollary}\label{(E)_for_proj_bands}
Archimedean pervasive pre-Riesz spaces have the extension property (E) for projection bands.
\end{corollary}

In the next example we see that the converse of Theorem~\ref{extideal_extband_coincide}\ref{extideal_extband_coincide.it1} is not true, in general.
\begin{example}\label{projband_not_restrictable}
\textit{A restriction of a projection band in a vector lattice cover need not be a projection band in the corresponding pre-Riesz space, even if the pre-Riesz space is a vector lattice.}

We call a function $x\colon [0,1]\rightarrow \RR$ \textit{piecewise right continuous}, if there exists an $n\in\NN$ and $t_0,\ldots,t_n\in [0,1]$ with $0=t_0<t_1<\ldots < t_n=1$ such that $x$ is continuous on $]t_k,t_{k+1}[$ and continuous from the right in $t_k$ for every $k\in\left\{0,n-1\right\}$.
Let $Y$ be the vector lattice of bounded piecewise right continuous functions on the interval $[0,1]$. The vector lattice $X:=C([0,1])$ of continuous functions is order dense in $Y$. That is, we can view $(Y,i)$ as a vector lattice cover of $X$, where $i$ is the identity embedding map. For the functions $y_1:=\1_{\left[0,\tfrac{1}{2}\right[}$ and $y_2:=\1_{\left[\tfrac{1}{2},1\right]}$ we have $y_1,y_2\in Y$. Consider the band $\hat{B}:=\B_{y_1}=\left\{y\in Y\mid| y(\left[\tfrac{1}{2},1\right])=0\right\}$ generated by $y_1$ in $Y$. Clearly, $Y=\hat{B}\oplus \hat{B}^{\t{d}}$ and $y_2\in \hat{B}^{\t{d}}$. Since $X$ is a vector lattice, $X$ is pervasive. Thus we have the restriction property for bands. It follows that $B:=\hat{B}\cap X$ and $B^{\t{d}}=\hat{B}^{\t{d}}\cap X$ are bands in $X$. However, we have $X\neq B\oplus B^{\t{d}}$. Indeed, for the element $x:=y_1+y_2=\1_{[0,1]}\in X$ there do not exist two continuous functions $x_1\in B$ and $x_2\in B^{\t{d}}$ such that $x=x_1+x_2$.

Notice that $B$ is not majorizing in $\hat{B}$. Clearly, the vector lattice $X$ has the RDP.
\end{example}

The following example shows that in the setting of Theorem~\ref{extideal_extband_coincide}\ref{extideal_extband_coincide.it1} the converse is not true, even if $B$ is majorizing in $\hat{B}$.
\begin{example}\label{projband_not_restrictable_noRDP}
\textit{A restriction $B$ of a projection band $\hat{B}$ in a vector lattice cover need not be a projection band in the corresponding pre-Riesz space, even if $i(B)$ is majorizing in $\hat{B}$.}

Let $PA[-1,1]$ be the vector lattice of continuous piecewise affine functions on the interval $[-1,1]$ and define $q(t):=t^2$ for every $t\in[-1,1]$. Let
\[X:=\Span\big(\left\{x\in PA[-1,1] \mid| x(0)=0\right\}\cup\left\{q\right\}\big)\]
be endowed with pointwise order.
Then $X$ is directed and Archimedean and therefore a pre-Riesz space.
Consider the sublattice $Y:=\left\{y\in C[-1,1]\mid| \exists\hs x \in X\colon |y|\leq x\right\}$ of $C[-1,1]$. 
It is immediate that $X$ is pervasive in $Y$, since for every positive $y\in Y$ there is a non-zero function $x\in PA[-1,1]$ with $x(0)=0$ such that $0<x\leq y$. Moreover, $X$ is majorizing in $Y$. It follows that $Y$ is a vector lattice cover of $X$.
We show that $X$ does not have the RDP. Indeed, define two functions $x_1,x_2$ on $[-1,1]$ by $x_1(t):=-t$ and $x_2(t):=0$ for $t\in[-1,0]$ and $x_1(t):=0$ and $x_2(t):=t$ for $t\in[0,1]$. Then $x_1,x_2\in X_+$ and $q\leq x_1+x_2$. However, there exist no elements $z_1,z_2\in X_+$ with $q=z_1+z_2$ such that $0\leq z_1\leq x_1$  and $0\leq z_2\leq x_2$.

Define two functions $q_1,q_2\in Y$ by $q_1(t):=t^2$ and $q_2(t):=0$ for $t\in[-1,0]$ and by $q_1(t):=0$ and $q_2(t):=t^2$ for $t\in[0,1]$. Then in $Y$ we have
\[\hat{B}:=\B_{q_1}=\left\{x\in Y \mid| x([0,1])=\left\{0\right\}\right\} \hs\t{ and }\hs \hat{B}^{\t{d}} = \B_{q_2}=\left\{x\in Y \mid| x([-1,0])=\left\{0\right\}\right\}.\]
It is immediate that $Y=\hat{B}\oplus\hat{B}^{\t{d}}$. Since $X$ is pervasive, it has the band restriction property. It follows that $B:=[\hat{B}]i$ is a band in $X$. Notice that $[\hat{B}^{\t{d}}]i=B^{\t{d}}$.

We have $X\neq B\oplus B^{\t{d}}$. Indeed, for $q\in X$ there do not exist $z_1\in B$ and $z_2\in B^{\t{d}}$ such that $q=z_1+z_2$.

Notice that in this example, $B$ is majorizing in $\hat{B}$.
\end{example}

In Examples~\ref{projband_not_restrictable} and \ref{projband_not_restrictable_noRDP} we saw that the conditions in the following theorem can not be omitted.
\begin{theorem}\label{projband_restrictable}
Let $X$ be an Archimedean fordable pre-Riesz space with RDP and $(Y,i)$ a vector lattice cover of $X$. Let $\hat{B}$ be a band in $Y$ such that $Y=\hat{B}\oplus\hat{B}^{\t{d}}$ and let $B:=[\hat{B}]i$. Moreover, let $i(B)$ be majorizing in $\hat{B}$ and $i(B^{\t{d}})$ be majorizing in $\hat{B}^{\t{d}}$. Then $X=B\oplus B^{\t{d}}$.
\end{theorem}
\begin{proof}
Since $X$ is fordable, $X$ has the restriction property for bands. Thus $[\hat{B}]i=B$ is a band in $X$. Let $x\in X_+$. Due to $i(x)\in Y=\hat{B}\oplus\hat{B}^{\t{d}}$ there exist $y_1\in\hat{B}$ and $y_2 \in\hat{B}^{\t{d}}$ such that $i(x)=y_1+y_2$. 
By Lemma~\ref{prelim.0}~\ref{prelim.0.it2} we have $y_1,y_2\geq 0$. Since $i(B)$ is majorizing in $\hat{B}$ and $i(B^{\t{d}})$ is majorizing in $\hat{B}^{\t{d}}$, there exist $b_1\in B$ and $b_2\in B^{\t{d}}$ such that $i(x) \leq i(b_1)+ i(b_2)$. That is, we have $x\leq b_1+b_2$ and $b_1,b_2\geq 0$. Since $X$ has the RDP, there exist $x_1,x_2\in X$ with $x=x_1+x_2$ such that $0\leq x_1\leq b_1$ and $0\leq x_2\leq b_2$. Since $B$ and $B^{\t{d}}$ are ideals in $X$, it follows $x_1\in B$ and $x_2\in B^{\t{d}}$.

We show the uniqueness of this decomposition. Let $x\in X_+$ be such that $x=x_1+x_2 =\tilde{x}_1 + \tilde{x}_2$, where $x_1,\tilde{x}_1,x_2,\tilde{x}_2\in X_+$ with $x_1,\tilde{x}_1\in B$ and $x_2,\tilde{x}_2\in B^{\t{d}}$. Then $x_1-\tilde{x}_1\in B$ and $x_1-\tilde{x}_1 = \tilde{x}_2-x_2\in B^{\t{d}}$. Due to $B\cap B^{\t{d}}=\left\{0\right\}$ it follows $x_1-\tilde{x}_1 = 0 = \tilde{x}_2-x_2$, i.e.\ $x_1=\tilde{x}_1$ and $x_2=\tilde{x}_2$.

Now let $x\in X$. As $X$ is directed, there exist $x^{(1)},x^{(2)}\in X_+$ such that $x=x^{(1)}+x^{(2)}$. Since by the first part of the proof both $x^{(1)}$ and $x^{(2)}$ have disjoint decompositions $x^{(1)}=x^{(1)}_1+x^{(1)}_2$ and $x^{(2)}=x^{(2)}_1+x^{(2)}_2$, it follows $x=(x^{(1)}_1 + x^{(2)}_1) - (x^{(1)}_2 + x^{(2)}_2)$, where $x^{(1)}_1 + x^{(2)}_1\in B$ and $x^{(1)}_2 + x^{(2)}_2\in B^{\t{d}}$.
\end{proof}

\section{Atoms and discrete elements in pre-Riesz spaces}\label{atoms}

In the definition of an atom in a vector lattice $X$ the inequality $|x| \leq |a|$, where $a,x\in X$, is equivalent to $\left\{a,-a\right\}^{u}\sub \left\{x,-x\right\}^{u}$. For pervasive pre-Riesz spaces we make the following observation.

\begin{proposition}\label{atomic.3}
Let $X$ be a pervasive pre-Riesz space and $a\in X\ohne\left\{0\right\}$. Assume that for every $x\in X$ with $\left\{a,-a\right\}^{u}\sub \left\{x,-x\right\}^{u}$ it follows $x=\lambda a$ for some $\lambda\in\RR$. Then we have $a>0$ or $a<0$.
\end{proposition}
\begin{proof}
Let $(Y,i)$ be a vector lattice cover of $X$. Since $a\neq 0$, we have $|i(a)|>0$.
As $X$ is pervasive, there exists an element $x\in X$ such that $0<i(x)\leq |i(a)|$. Due to $x>0$ it follows $\left\{a,-a\right\}^{u}\sub \left\{x\right\}^u= \left\{x,-x\right\}^{u}$. By assumption there is $\lambda\in\RR$ with $x=\lambda a$ in $X$.
Due to $x> 0$ the case $\lambda>0$ leads to $a>0$ and the case $\lambda < 0$ leads to $a<0$.
\end{proof}

In view of Proposition~\ref{atomic.3} we define atoms in the following way. The definition\footnote{We stress that in \c[Definition~1.42]{AliTou} the authors introduce this concept using a different term. They call an element $a\in X_+$ for which $0\leq x\leq a$ implies that $x=\lambda a$ for some real $\lambda \geq 0$ an \emph{extremal vector} or a \emph{discrete vector} of $X_+$. However, a similar concept is well-known in the less general setting of vector lattices. Indeed, \c[\S~3]{Schae} gives a definition of an atom in a vector lattice which differs from our notion by the fact that the considered element need not be positive. However, by \c[III.13.1~a)]{Vulikh_en} for every atom $a\in Y$ it follows $a>0$ or $a<0$. Moreover, if $a<0$ is an atom, so is $-a$. These circumstances maybe clarify that, based on \c[Definition~1.42]{AliTou}, we define atoms as positive elements. Proposition~\ref{atomic.3} justifies our choice in Archimedean pervasive pre-Riesz spaces.} is based on \c[Definition~1.42]{AliTou}.
\begin{definition}\label{prelim.1a}
Let $X$ be an ordered vector space with the cone $X_+$. An element $a\in X_+\ohne\left\{0\right\}$ is said to be an \textbf{atom}
if $0\leq x\leq a$ implies that $x=\lambda a$ for some real $\lambda \geq 0$.
For the \textbf{set of all atoms} in $X$ we write
$\mathcal{A}_X:=\left\{a\in X_+\mid| a \textit{ is an atom}\right\}$.
\end{definition}

The following characterization of atoms is from \c[Lemma~1.43]{AliTou}.
\begin{proposition}\label{prelim.1b}
Let $X$ be an ordered vector space. For a non-zero vector $a\in X_+$ the following statements are equivalent.
\begin{enumerate}
\itemsep0em
\item The vector $a$ is an atom in $X$.
\item If $x_1,x_2\in X_+$ satisfy $a=x_1+x_2$, then the vectors $x_1$ and $x_2$ are linearly dependent.
\item The half-ray $\left\{\lambda a \mid| \lambda\in\RR_{\geq 0}\right\}$ is a face of $X_+$.
\end{enumerate}
\end{proposition}

We already defined discrete elements in vector lattices. Based on this definition and by courtesy of the disjointness notion in ordered vector spaces given in \eqref{disjoint} we can generalize the term \emph{discrete element} to ordered vector spaces.
\begin{definition}\label{atomic.2}
Let $X$ be an ordered vector space.
An element $a\in X$, $a> 0$, is called a \textbf{discrete element}, if for every $x,z\in X_+$ with $x, z \leq a$ and $x\perp z$ it follows $x=0$ or $z=0$.
\end{definition}

In vector lattices, discrete elements and atoms are related notions.
In a (not necessarily Archimedean) vector lattice $Y$ every atom is a discrete element by \c[Lemma~26.2~(i)]{Zaa1}, but discrete elements need not be atoms, as one can see in \c[Lemma~26.2~(ii)]{Zaa1}. However, if $Y$ is an Archimedean vector lattice, then by \c[III.13.1~b)]{Vulikh_en} and by \c[Lemma~26.2~(ii)]{Zaa1} the notions of an atom and of a discrete element are equivalent. Due to this fact, in the vector lattice theory these two concepts are used interchangeably\footnote{For instance, in \c[Definition~26.1]{Zaa1} the terms \emph{atom} and \emph{discrete element} are used in a reversed way with respect to our terminology.}. We establish in Theorem~\ref{atomic.4} below that in Archimedean pervasive pre-Riesz spaces the notions of an atom and of a discrete element coincide as well. 
However, we see in Proposition~\ref{atomic.3u} and in Example~\ref{atomic.4u} below that in an Archimedean (not necessarily pervasive) pre-Riesz space every atom is discrete, but discrete elements need not be atoms.
\begin{proposition}\label{atomic.3u}
Let $X$ be a pre-Riesz space and $a\in \mathcal{A}_X$. Then $a$ is discrete.
\end{proposition}
\begin{proof}
Let $a\in\mathcal{A}_X$ and let elements $x,z\in X_+$ satisfy the inequalities $x,z\leq a$ and $x\perp z$. Since $a$ is an atom, there are positive real numbers $\lambda$ and $\mu$ such that $x=\lambda a$ and $z=\mu a$. Let $(Y,i)$ be a vector lattice cover of $X$. By Proposition~\ref{prelim.36} the relation $x\perp z$ implies that $i(x)\perp i(z)$, i.e.\ we have $\lambda i(a) \perp \mu i(a)$ with $i(a)\in Y_+\ohne\left\{0\right\}$. This yields $\lambda=0$ or $\mu=0$, i.e.\ $x=0$ or $z=0$. Thus $a$ is a discrete element in $X$.
\end{proof}

There are pre-Riesz spaces that contain no non-trivial disjoint elements. Consider, e.g.\ $\RR^3$ endowed with the ice-cream cone, see \c[Proposition~16]{5} and \c[Examples~4.5 and 4.6]{2}. Clearly, in such a pre-Riesz space every element $x>0$ is a discrete element, whereas the atoms are precisely the non-zero elements on the boundary of the cone. We give an example of a finite-dimensional space which contains non-trivial disjoint elements in the cone and where a discrete element need not be an atom. 

\begin{example}\label{atomic.4u}
\textit{In an Archimedean pre-Riesz space with an order unit a discrete element need not be an atom.}

Let $X:=\RR^3$ be the three-dimensional Euclidean space. Consider in $X$ the four vectors
\[v_1=\smallmath[.7]{\left(\begin{array}{c} 1\\ 0\\ 1 \end{array}\right)},\hs
	v_2=\smallmath[.7]{\left(\begin{array}{c} 0\\ 1\\ 1 \end{array}\right)},\hs
	v_3=\smallmath[.7]{\left(\begin{array}{c} -1\\ 0\\ 1 \end{array}\right)},\hs
	v_4=\smallmath[.7]{\left(\begin{array}{c} 0\\ -1\\ 1 \end{array}\right)},
\]
and let the order on $X$ be induced by the \textit{four-ray-cone} $K_4:=\t{pos}\left\{ v_1,v_2,v_3,v_4\right\}$, i.e.\ by the positive-linear hull of the vectors $v_1,\ldots,v_4$. Clearly, $(X, K_4)$ has an order unit.
In \c[Example~4.8]{2} it is shown that $(X,K_4)$ can be order densely embedded into $\RR^4$ with the standard cone $\RR^4_+$. Indeed, consider the four functionals $f_i\in X'$ on $X$, where $X'$ is identified with $\RR^3$:
\[f_1=\smallmath[.7]{\left(\begin{array}{c} -1\\ -1\\ 1 \end{array}\right)},\hs 
	f_2=\smallmath[.7]{\left(\begin{array}{c} 1\\ -1\\ 1 \end{array}\right)},\hs
	f_3=\smallmath[.7]{\left(\begin{array}{c} 1\\ 1\\ 1 \end{array}\right)},\hs
	f_4=\smallmath[.7]{\left(\begin{array}{c} -1\\ 1\\ 1 \end{array}\right)}.
\]
Then the map $i\colon X\rightarrow \RR^4$, given by
$i: x\mapsto (f_1(x),\ldots,f_4(x))^T$,
is a bipositive embedding of the pre-Riesz space $(X, K_4)$ into $(\RR^4,\RR^4_+)$. Notice that by \c[Proposition~13]{5} the space $(\RR^4,\RR^4_+)$ is the Riesz completion of $(X, K_4)$.
The pre-Riesz space $(X, K_4)$ is not pervasive, see \c[Example~4.4.18]{PRS}.

By Proposition~\ref{prelim.1b} an element $a\in X_+$ is an atom if and only if the half-ray $\left\{\lambda a\mid| \lambda\in\RR_{\geq 0}\right\}$ is a face of the cone $K_4$. Hence, $\mathcal{A}_X=\left\{\lambda v_k \mid| \lambda\in\RR_{>0}, k\in\left\{1,\ldots,4\right\}\right\}$. 
Let $d:=v_1+v_2$. Clearly, $d\notin\mathcal{A}_X$. However, $d$ is a discrete element. Indeed, by \c[Example~4.6]{1}, we have
$\left\{v_1\right\}^{\t{d}} \hspace*{-.5mm}=\hspace*{-.2mm} \Span\left\{v_3\right\}$,
$\left\{v_2\right\}^{\t{d}} \hspace*{-.5mm}=\hspace*{-.2mm} \Span\left\{v_4\right\}$,
$\left\{v_3\right\}^{\t{d}} \hspace*{-.5mm}=\hspace*{-.2mm} \Span\left\{v_1\right\}$,
$\left\{v_4\right\}^{\t{d}} \hspace*{-.5mm}=\hspace*{-.2mm} \Span\left\{v_2\right\}$ and
$\left\{x\right\}^{\t{d}} \hspace*{-.5mm}=\hspace*{-.2mm} \left\{0\right\}$ for every $x\in K_4\ohne\left(\Span\left\{v_1\right\}\cup\Span\left\{v_2\right\}\cup\Span\left\{v_3\right\}\cup\Span\left\{v_4\right\}\right)$.
Since $[0,d]$ is the convex hull of the set $\left\{0,v_1,v_2,d\right\}$,
it follows that there are no disjoint elements $x,z\in X_+\ohne\left\{0\right\}$ such that $x,z\leq d$. Therefore $d$ is discrete.
\end{example}

The next result yields for a pervasive pre-Riesz space $X$ relationships between atoms in $X$ and in its vector lattice cover.
\begin{theorem}\label{atomic.4}
Let $X$ be an Archimedean pervasive pre-Riesz space and $(Y,i)$ a vector lattice cover of $X$. Let $a\in X$.
\begin{enumerate}
\itemsep0em
\item\label{atomic.4.eq1} We have $a\in\mathcal{A}_X$ if and only if $i(a)\in\mathcal{A}_Y$.
\item\label{atomic.4.eq2} The element $a$ is discrete in $X$ if and only if $i(a)$ is discrete in $Y$.
\item\label{atomic.4.eq3} We have $a\in\mathcal{A}_X$ if and only if $a$ is a discrete element.
\item\label{atomic.4.eq4} If $a\in\mathcal{A}_X$, then we have $i(\I_a) = \I_{i(a)}$.
\item\label{atomic.4.eq5} If $\tilde{a}\in\mathcal{A}_Y$, then $\tilde{a}\in i(X)$ and $i^{-1}(\tilde{a})\in\mathcal{A}_X$.
\end{enumerate}
\end{theorem}
\begin{proof}
\ref{atomic.4.eq1}: 
Let $a\in\mathcal{A}_X$. Let $y\in Y$ be such that $0< y\leq i(a)$. Since $X$ is pervasive, by Proposition~\ref{properties.1} we have for the set $S:=\left\{z\in i(X) \mid| 0< z \leq y\right\}$ the relation $y =\sup S$. For every $z\in S$ it holds that $0< z \leq y \leq i(a)$, i.e.\ for $x:=i^{-1}(z)$ we have $0< x\leq a$. Since $a\in X$ is an atom, there exists a real $\lambda > 0$ such that $x=\lambda a$. Therefore we have
$S=\left\{\lambda i(a) \mid| 0< \lambda i(a) \leq y\right\}$.
Due to $Y$ being Archimedean, the set
$\Lambda:=\left\{\lambda \in\RR \mid| 0< \lambda i(a) \leq y\right\}$
is bounded above in $\RR$ and therefore its supremum $\sup \Lambda\geq 0$ exists in $\RR$.
For every $\lambda\in \Lambda$ we have $\lambda i(a) \leq y$, thus due to Proposition~\ref{prelim.17bb} taking the supremum leads to $(\sup \Lambda) i(a) \leq y = \sup S$. On the other hand, for every $\lambda i(a)\in S$ we have $\lambda i(a)\leq (\sup \Lambda) i(a)$ and therefore $\sup S \leq (\sup \Lambda) i(a)$. It follows $y=\sup S = (\sup \Lambda) i(a)$, i.e. $i(a)$ is an atom in $Y$.

Conversely, let $a\in X$ be such that $i(a)\in \mathcal{A}_Y$. Let $x\in X$ with $0< x\leq a$. As $i(a)$ is an atom, $0<i(x)\leq i(a)$ implies that there is $\lambda\in\RR$ with $i(x) = \lambda i(a)$. It follows $x = \lambda a$, i.e. $a$ is an atom in $X$.

\ref{atomic.4.eq2}: Let $a\in X$ be a discrete element in $X$. We show that $i(a)$ is a discrete element in $Y$.
Let $y_1,y_2\in Y$ be such that $0\leq y_1, y_2\leq i(a)$ and $y_1\perp y_2$. As $X$ is pervasive, by Proposition~\ref{properties.1} for the sets
\[S_n:=\left\{z\in i(X) \mid| 0\leq z \leq y_n\right\} \quad\quad (n\in\left\{1,2\right\})\]
we have the relations $y_n= \sup S_n$. Due to $y_1\perp y_2$ we obtain for every $x_1\in X$ with $i(x_1)\in S_1$ and every $x_2\in X$ with $i(x_2)\in S_2$ the relations $x_1\perp x_2$ and $0\leq x_1, x_2 \leq a$.
Assume $y_1\neq 0$, then there is $x\in X_+\ohne\left\{0\right\}$ with $i(x)\in S_1$. As $a$ is discrete and $i(x)\neq 0$, it follows for every $x_2\in X$ with $i(x_2)\in S_2$ that $x_2=0$, i.e.\ $S_2=\left\{0\right\}$. Due to $y_2=\sup S_2$ we conclude $y_2=0$. Therefore $i(a)$ is discrete.

Conversely, let $i(a)$ be discrete. Let $x_1,x_2\in X$ be such that $0\leq x_1,x_2\leq a$ and $x_1\perp x_2$. For $y_1:=i(x_1)$ and $y_2:=i(x_2)$ we have $0\leq y_1,y_2\leq i(a)$ and $y_1\perp y_2$ by Proposition~\ref{prelim.36}. As $i(a)$ is discrete, it follows $y_1 =0$ or $y_2=0$, i.e.\ $x_1=0$ or $x_2=0$. That is, $a$ is discrete.

\ref{atomic.4.eq3}: 
If $a\in \mathcal{A}_X$, then by Proposition~\ref{atomic.3u} the element $a$ is discrete.
Conversely, let $a\in X$ be a discrete element. By \ref{atomic.4.eq2} the element $i(a)$ is discrete in $Y$ as well. By Proposition~\ref{prelim.19.2} every discrete element in an Archimedean vector lattice is an atom, i.e\ $i(a)\in\mathcal{A}_Y$. By \ref{atomic.4.eq1} it follows $a\in\mathcal{A}_X$.

\ref{atomic.4.eq4}: Let $a\in\mathcal{A}_X$. The inclusion $i(\I_a) \sub \I_{i(a)}$ is immediate. By \ref{atomic.4.eq1} we have $i(a)\in\mathcal{A}_Y$. For every $y\in \I_{i(a)}$ it follows that there exists $\lambda\in\RR$ such that $y=\lambda i(a)\in i(\I_a)$. Thus we have $i(\I_a) \supseteq \I_{i(a)}$. We conclude $i(\I_a) =\I_{i(a)}$.

\ref{atomic.4.eq5}: Let $\tilde{a}\in\mathcal{A}_Y$. Since $X$ is pervasive, there exists $x\in X$ such that $0< i(x) \leq \tilde{a}$. As $\tilde{a}$ is an atom in $Y$, there exists $\lambda\in\RR$ such that $\tilde{a}=\lambda i(x)=i(\lambda x) \in \mathcal{A}_Y$. By \ref{atomic.4.eq1} it follows $a:=i^{-1}(\tilde{a})=\lambda x\in\mathcal{A}_X$.
\end{proof}

Notice that in part \ref{atomic.4.eq1} of the above result pervasiveness of $X$ is not needed for the implication ``$\Leftarrow$''. This leads to the following observation.
\begin{proposition}\label{atomic.5}
Let $X$ be an Archimedean pre-Riesz space and $Y$ a vector lattice cover of $X$. Then we have $\mathcal{A}_Y\cap i(X) \sub i(\mathcal{A}_X)$.
\end{proposition}

\begin{remark}
\begin{enumerate}
\itemsep0em
\item
We can reformulate \ref{atomic.4.eq5} of Theorem~\ref{atomic.4} as $\mathcal{A}_Y\sub i(\mathcal{A}_X)$.
\item
If $X$ is not Archimedean, then the equivalence in Theorem~\ref{atomic.4} \ref{atomic.4.eq3} is not true, in general. As already mentioned, without the condition of Archimedeanity even in a vector lattice a discrete element need not be an atom.
\end{enumerate}
\end{remark}

\section{Finite-dimensional pervasive pre-Riesz spaces}\label{fin_dim}
With the help of atoms, we now investigate finite-dimensional Archimedean pervasive pre-Riesz spaces. In Theorem~\ref{atomic.101a} below we establish that all such spaces are, in fact, vector lattices. We start with two technical statements.
\begin{proposition}\label{atomic.100}
Let $X$ be an Archimedean pre-Riesz space and $a_1,a_2\in \mathcal{A}_X$. 
\begin{enumerate}
\itemsep0em
\item\label{atomic.100.it1} If $a_1\perp a_2$, then $a_1$ and $a_2$ are linearly independent.
\item\label{atomic.100.it2} Let $X$ be additionally pervasive. If $a_1$ and $a_2$ are linearly independent, then $a_1\perp a_2$.
\end{enumerate}
\end{proposition}
\begin{proof}
\ref{atomic.100.it1}: Let $a_1,a_2\in \mathcal{A}_X$ with $a_1\perp a_2$. Then in a vector lattice cover $(Y,i)$ of $X$ we have by Proposition~\ref{prelim.36} that $i(a_1)\perp i(a_2)$. That is, $i(a_1)\Inf i(a_2)=0$. By Proposition~\ref{properties.23y}~\ref{properties.23y.it2} the infimum $a_1\Inf a_2$ exists in $X$ and equals $0$. this implies that $a_1$ and $a_2$ are linearly independent. Indeed, let on the contrary, $a_1=\lambda a_2$ for some $\lambda\in\RR_{>0}$. Then $0=a_1\Inf a_2 = a_1\Inf (\lambda a_1) = \min\left\{1,\lambda\right\} a_1>0$, a contradiction.

\ref{atomic.100.it2}: Let $a_1,a_2\in \mathcal{A}_X$. We show that if $a_1\not\perp a_2$, then $a_1,a_2$ are linearly dependent.

Let $a_1\not\perp a_2$, i.e.\ $i(a_1)\Inf i(a_2)>0$. Since $X$ is pervasive, there exists an element $a\in X$ with $0< i(a)\leq i(a_1)\Inf i(a_2)$, i.e. $0<a \leq a_1$ and $0<a\leq a_2$. Since $a_1,a_2$ are atoms, there exist $\lambda_1,\lambda_2\in\RR_{>0}$ with $\lambda_1 a_1 =a= \lambda_2 a_2$, i.e.\ $a_1$ and $a_2$ are linearly dependent.
\end{proof}

In Proposition~\ref{atomic.100} the condition of $X$ being pervasive can not be omitted. Indeed, in the Archimedean non-pervasive pre-Riesz space in Example~\ref{atomic.4u} the linearly independent vectors $v_1$ and $v_2$ are atoms, but they are not disjoint.

The idea for the next result originated during a discussion with H. van Imhoff.
\begin{proposition}\label{atomic.101}
Let $X$ be an Archimedean pervasive pre-Riesz space and the atoms $a_1,\ldots,a_n\in \mathcal{A}_X$ pairwise linearly independent. Then the set $\left\{a_1,\ldots, a_n\right\}$ of atoms is linearly independent.
\end{proposition}
\begin{proof}
Assume that the set $\left\{a_1,\ldots,a_n\right\}\sub \mathcal{A}_X$ of pairwise linearly independent atoms is linearly dependent, that is, there exist $\lambda_k\in\RR$ for $k\in\left\{2,\ldots,n\right\}$ such that
\[a_1 =\sum_{k=2}^n \lambda_k a_k.\]
Due to Proposition~\ref{atomic.100} the atoms $a_1,\ldots,a_n$ are pairwise disjoint. Thus in a vector lattice cover $(Y,i)$ of $X$ for $k,j\in\left\{1,\ldots,n\right\}$ with $k\neq j$ we have $i(a_k)\perp \lambda_j i(a_j)$. Therefore we can apply Lemma~\ref{prelim.0}\ref{prelim.0.it1} and obtain the following contradiction:
\[0 \hs<\hs i(a_1) \hs=\hs i(a_1) \Inf \sum_{k=2}^n \lambda_k i(a_k) \hs=\hs \sum_{k=2}^n (i(a_1)\Inf \lambda_k i(a_k)) \hs=\hs 0.\]
\end{proof}

Observe that the conclusion of Proposition~\ref{atomic.101} is not true if the pre-Riesz space is not pervasive, see \c[p.~38]{AliTou}.

To characterize finite-dimensional pervasive pre-Riesz spaces in Theorem~\ref{atomic.101a} below, we first recall some basics. Let $X$ be an ordered vector space. A convex set $B\sub X_+$ is called a \emph{base} of the cone $X_+$ if every $x\in X_+\ohne\left\{0\right\}$ has a unique representation $x=\lambda b$, where $\lambda\in\RR_{>0}$ and $b\in B$.  
The following result can be found in \c[Theorem~1.48]{AliTou}.
\begin{proposition}\label{atomsareextreme}
Let $X$ be an ordered vector space such that the cone $X_+$ has a base $B$. Let $b\in B$. Then $b\in\mathcal{A}_X$ if and only if $b$ is an extreme point of $B$.
\end{proposition}
From \c[Theorems~IV.1.1, VII.1.1]{VulikhWeber} we obtain the following. 
\begin{proposition}\label{conestuff}
Let $X$ be a finite-dimensional ordered Banach space. If the cone $X_+$ is (norm) closed, then $X_+$ has a norm bounded base.
\end{proposition}

\begin{theorem}\label{atomic.101a}
Let $X$ be an $n$-dimensional Archimedean pre-Riesz space. Then $X$ is pervasive if and only if $X$ is a vector lattice.
\end{theorem}
\begin{proof}
Clearly, if $X$ is a vector lattice, then it is pervasive in its Riesz completion $X^\rho=X$.
Conversely, let $X$ be an $n$-dimensional Archimedean pervasive pre-Riesz space. Endowed with the Euclidean norm, $X$ is an ordered Banach space.
Since $X$ is Archimedean, the cone $X_+$ is closed. By Proposition~\ref{conestuff} the cone $X_+$ has a norm bounded base $B$. Let $\mathcal{A}:=\mathcal{A}_X\cap B$.
We first establish $X_+=\textnormal{pos}\mathcal{A}$. The base $B$ is a convex set and, as $X$ is finite-dimensional, also compact. By Proposition~\ref{atomsareextreme} a vector $a\in B$ is an extreme point of $B$ if and only if $a$ is an atom. By Minkowski's theorem\footnote{Minkowski's theorem states that in a finite-dimensional normed vector space every element of a compact convex set $S\sub X$ is a convex combination of extreme points of $S$, see e.g.\ \c[p.~1]{Phelps}.} every $b\in B$ is a convex combination of atoms in $\mathcal{A}$. From $X_+=\bigcup_{\lambda\in\RR_{\geq 0}}\lambda B$, it follows $X_+=\textnormal{pos}\mathcal{A}$.
Observe that $\mathcal{A}$ contains pairwise linearly independent atoms. Indeed, if $a,b\in\mathcal{A}$ with $a\neq b$ such that $a=\lambda b$, then $a$ has two representations with respect to the basis, which is a contradiction.

Next we consider three cases for the cardinality of $\mathcal{A}$. If $|\mathcal{A}|<n$, then $X_+$ is polyhedral. As $X_+=\textnormal{pos}\mathcal{A}$, the cone is not generating and the space is not even pre-Riesz, a contradiction. Let $|\mathcal{A}|>n$ and pick $m>n$ atoms $a_1,\ldots,a_m\in\mathcal{A}$. Then due to $a_1,\ldots, a_m$ being pairwise linearly independent, by Proposition~\ref{atomic.101} the set of atoms $\left\{a_1,\ldots, a_m\right\}$ is linearly independent. But the number $m$ of linearly independent elements is greater than the dimension $n$ of the space, a contradiction. Finally, let $|\mathcal{A}|=n$, i.e.\ $X_+$ has $n$ extreme rays. Then $(X,X_+)$ is order isomorphic to $(\RR^n,\RR^n_+)$. By Proposition~\ref{Yudin} it follows that $X$ is a vector lattice.
\end{proof}

The combination of Proposition~\ref{Yudin} and Theorem~\ref{atomic.101a} leads to the following.
\begin{corollary}
Let $(X,K)$ be an $n$-dimensional Archimedean pre-Riesz space. Each of the properties \ref{Yudin.it1} to \ref{Yudin.it3} of Proposition~\ref{Yudin} is equivalent to
\begin{enumerate}
\itemsep0em
\item[\textit{(iv)}] $(X,K)$ is pervasive.
\end{enumerate}
\end{corollary}

\section{Atoms and principal bands in pre-Riesz spaces}\label{principal_bands}

The following properties of atoms were partially shown for vector lattices in \c[Proposition~III.13]{Vulikh_en}, see also Proposition~\ref{prelim.19.6}. We generalize them for pervasive pre-Riesz spaces in the following technical result.
\begin{lemma}\label{atomic.7}
Let $X$ be an Archimedean pervasive pre-Riesz space. Let $a\in \mathcal{A}_X$ and $x\in X_+$. Then we have the following.
\begin{enumerate}
\itemsep0em
\item\label{atomic.7.eq1} There exists $\lambda\in\RR_{\geq 0}$ such that $x-\lambda a\geq 0$ and $(x-\lambda a)\perp a$.
\item\label{atomic.7.eq2} The real number $\lambda\in\RR$ in \ref{atomic.7.eq1} is unique. Moreover,
\begin{enumerate}
	\itemsep0em
	\item\label{atomic.7.eq3} $\lambda a \leq x$ and for every $\eps>0$ we have $(\lambda +\eps)a \not\leq x$, i.e.\ $\lambda$ is the greatest number $\mu\in\RR$ which satisfies the inequality $\mu a \leq x$,
	\item\label{atomic.7.eq4} $\lambda =0$ holds if and only if $x\perp a$.
\end{enumerate}
\end{enumerate}
\end{lemma}
\begin{proof}
Let $(Y,i)$ be a vector lattice cover of $X$.
We consider four different cases and first show \ref{atomic.7.eq1} and \ref{atomic.7.eq2}\ref{atomic.7.eq3} in each of them.

\underline{Case 1}: Let first $0<a\leq x$, i.e.\ $0<i(a) \leq i(x)$.

\ref{atomic.7.eq1}: Since $a$ is an atom, by Theorem~\ref{atomic.4}~\ref{atomic.4.eq1} it follows that $i(a)$ is an atom in $Y$. By Proposition~\ref{prelim.19.6} there exists $\lambda\in\RR_{>0}$ such that $i(x)-\lambda\hs i(a) \geq 0$ and $(i(x)-\lambda\hs i(a))\perp i(a)$. It follows $x-\lambda a\geq 0$, and by Proposition~\ref{prelim.36} we have $(x-\lambda a)\perp a$ in $X$, which establishes \ref{atomic.7.eq1}.

\ref{atomic.7.eq2}: 
We prove the statement \ref{atomic.7.eq3}. Assume that there are two numbers $\lambda, \mu\in\RR_{\geq 0}$  with $\mu\neq \lambda$ such that $x-\lambda a\geq 0$, $(x-\lambda a)\perp a$ and $x-\mu a\geq 0$, $(x-\mu a)\perp a$. Let without loss of generality $\mu >\lambda$. Since $(x-\lambda a)\perp a$, it follows
\[[x -\mu a +(\mu -\lambda)a] \perp a,\]
or, equivalently, $0=[(i(x) -\mu i(a)) +(\mu -\lambda)i(a)] \Inf i(a)$. Due to $(x -\mu a)\perp(\mu -\lambda)a$ we have $(i(x) -\mu i(a))\perp(\mu -\lambda)i(a)$, from which (using Lemma~\ref{prelim.0}\ref{prelim.0.it1} in the second step) it follows
\begin{align*}
0&=[(i(x) -\mu i(a)) +(\mu -\lambda)i(a)] \Inf i(a)\\
&= [(i(x) -\mu i(a))\Inf i(a)] + [\left((\mu -\lambda)i(a)\right) \Inf i(a)] \\
&= [\left((\mu -\lambda)i(a)\right) \Inf i(a)] \neq 0,
\end{align*}
a contradiction. Thus the number $\lambda\in\RR$ with the properties as in \ref{atomic.7.eq1} is uniquely defined. 

Moreover, for every $\eps>0$ we have $(\lambda +\eps)a \not\leq x$. Indeed, let $\mu:=\lambda+\eps$ and assume $\mu a \leq x$. Then due to $0 \leq x-\mu a < x-\lambda a$ and $(x-\lambda a)\perp a$ it follows $(x-\mu a)\perp a$, i.e.\ $\mu\in\RR_{>0}$ with $\mu\neq \lambda$ is another real number which satisfies the conditions in \ref{atomic.7.eq1}, a contradiction to the uniqueness of $\lambda$. 

\underline{Case 2}: Let $0\leq x <a$.

\ref{atomic.7.eq1}: Due to $a$ being an atom there is $\lambda\in\RR_{\geq 0}$ with $x = \lambda a$. Thus $x-\lambda a = 0$ and therefore $(x-\lambda a)\Inf a =0$.

\ref{atomic.7.eq2}: Let $\eps>0$. Then we have $(\lambda+\eps)a >x$, that is $(\lambda+\eps)a \not\leq x$, i.e.\ there does not exist a real number greater than $\lambda$ satisfying \ref{atomic.7.eq1}. Moreover, if $0<\mu <\lambda$ for some $\mu\in\RR$, then we have $0<\mu a \leq x$. But since $x-\mu a = (\lambda -\mu)a>0$, it follows $(x-\mu a)\not\perp a$, i.e.\ $\lambda$ is the only real number which satisfies the properties in \ref{atomic.7.eq1}.

\underline{Case 3}: Let $a$ and $x$ be not comparable and let there exist $\gamma >0$ such that $0<\gamma a\leq x$. Then we can apply Case 1 to show \ref{atomic.7.eq1} and \ref{atomic.7.eq2}.

\underline{Case 4}: Finally, consider the case where $a$ and $x$ are not comparable, but there does not exist $\gamma >0$ such that $0<\gamma a\leq x$.

\ref{atomic.7.eq1}: Let $\lambda=0$. We show $(x-\lambda a) \perp a$, i.e.\ $x\perp a$, by contradiction. Assume that $x\not\perp a$. Then it follows $i(x)\Inf i(a)\neq 0$. Due to $X$ being pervasive there exists $\tilde{a}\in X$ such that $0<i(\tilde{a}) \leq i(x)\Inf i(a) \leq i(a)$. Since $a$ is an atom, by Theorem~\ref{atomic.4}~\ref{atomic.4.eq1} the element $i(a)$ is also an atom and therefore there exists $\gamma >0$ with $0< i(\tilde{a})= \gamma\hs i(a)$, i.e.\ $0<\tilde{a} = \gamma a$. Moreover, we have $0<i(\tilde{a}) \leq i(x)$, i.e. $0<\gamma a \leq x$, a contradiction.

\ref{atomic.7.eq2}: We have $x-\lambda a = x \geq 0$. Moreover, for $\gamma>0$ by assumption we obtain $0<(\lambda +\gamma) a = \gamma a\not\leq x$. This leads to the uniqueness of $\lambda$.

\ref{atomic.7.eq4} Notice that $\lambda =0$ holds only in Case 4. Due to the uniqueness of the number $\lambda$ in each of the four cases, Case 4 is the only case in which we have $x\perp a$. That is, $\lambda =0$ holds if and only if $x\perp a$.
\end{proof}

By Lemma~\ref{atomic.7} we can decompose every $x\in X$ 
as $x=\lambda a +(x-\lambda a)$, where the two summands are disjoint. In particular, one summand belongs to the band $\B_a$ generated by $a$. The next example depicts such a situation in a specific pre-Riesz space. Similar to the vector lattice case, we call an ordered vector space $X$ \emph{atomic}, if for every $x\in X$ with $x>0$ there is an atom $a\in\mathcal{A}_X$ such that $0<a\leq x$.
\begin{example}\label{atomic.6}
\textit{An Archimedean atomic pervasive pre-Riesz space $X$ which is not a vector lattice.}

For the space $C^1[0,1]$ of continuously differentiable functions on $[0,1]$ and characteristic functions $\1_{\{x\}}$ of some singleton $\{x\}\sub [0,1]$ we consider the space
\[X:=\Span\big(C^1[0,1]\cup\left\{\1_{\{x\}}\mid|x\in[0,1]\right\}\big)\]
with pointwise order. Clearly, $X$ is atomic with
$\mathcal{A}_X=\left\{\lambda\1_{\{x\}}\mid|x\in[0,1],\lambda\in\RR_{> 0}\right\}$.
Moreover, $X$ is Archimedean and directed, therefore $X$ is pre-Riesz. However, $X$ is not a vector lattice, as the two differentiable linear functions $x_1(t):=t$ and $x_2(t)=1-t$ have no infimum in $X$. In analogy to $C[0,1]$ being a vector lattice cover of $C^1[0,1]$, a vector lattice cover of $X$ is given by
\[Y:=\Span\big(C[0,1]\cup\left\{\1_{\{x\}}\mid|x\in[0,1]\right\}\big).\]
It is easy to see that $X$ is pervasive, since for every $y\in Y_+$, $y\neq 0$, we can find an atom $a\in X$ with $0<i(a)\leq y$. 

For every $a\in\mathcal{A}_X$ we have for the principal ideal $\I_a$ generated by $a$ and for the principal band $\B_a$ generated by $a$ the equality $\I_a=\B_a$ in $X$. Moreover, the band $\B_a$ is one-dimensional. 
\end{example}

\begin{theorem}\label{atomic.8}
Let $X$ be an Archimedean pervasive pre-Riesz space and $a\in \mathcal{A}_X$. Then the principal ideal and the principal band generated by $a$ coincide, i.e.\ $\I_a = \B_a$.
\end{theorem}
\begin{proof}
Let $a\in \mathcal{A}_X$. It follows that the ideal $\I_a$ is one-dimensional, i.e.\
$\I_a=\left\{\lambda a\mid| \lambda\in \RR\right\}$.
Consider the band $\B_a=\left\{a\right\}\dd$.
Assume that there exists an element $y\in \B_a\ohne \I_a$. That is, $y\in\B_a$ and for every $\lambda \in\RR$ we have $y\neq \lambda a$. We intend to apply Lemma~\ref{atomic.7}, for which we need an appropriate positive element.
Since $X$ is pervasive, by Proposition~\ref{properties.1} it follows
\[ |i(y)|=\sup\left\{z\in i(X)\mid| 0< z\leq |i(y)|\right\}.\]
Let $S:=\left\{z\in i(X)\mid| 0< z\leq |i(y)|\right\}$. First we show by contradiction that there exists a positive element 
$v\in S$ such that for $x:=i^{-1}(v)$ and for every $\lambda\in \RR$ we have $x\neq \lambda a$. Assume, on the contrary, that for every $z\in S$ there is $\lambda_z \in\RR_{\geq 0}$ such that $z=\lambda_z i(a)$. Since $Y$ is Archimedean, by Proposition~\ref{prelim.17bb} for $\lambda:=\sup\left\{\lambda_z\mid|z\in S\right\}\in\RR$ we have $|i(y)|=\lambda i(a)$.
This contradicts $y\notin\I_a$. Thus there exists $x\in X$ such that $i(x)\in S$, and for every $\lambda\in \RR$ we have $x\neq \lambda a$. Moreover, due to $0<i(x)\leq |i(y)|$ we have $\left\{x,-x\right\}^u \supseteq \left\{y,-y\right\}^u$. Due to $\B_a$ being an ideal, it follows $x\in \B_a$. We conclude that there is a positive element $x\in X$, $x\neq 0$, with $x\in \B_a$, such that for every $\lambda\in\RR$ we have $x\neq \lambda a$. 

Second, by Lemma~\ref{atomic.7} there exists $\lambda\in\RR_{\geq 0}$ such that $(x-\lambda a)\perp a$. Due to $x,\lambda a\in \B_a$ it follows $w:= x-\lambda a\in \B_a$. Thus $w\in \B_a=\left\{a\right\}\dd$, and, due to $w\perp a$, we get $w\in\left\{a\right\}^{\t{d}}$. We have $\left\{a\right\}\dd \cap \left\{a\right\}^{\t{d}}=\left\{0\right\}$, which implies $w=0$. This leads to $x=\lambda a$, a contradiction.

We established for every $y\in\B_a$ the relation $y= \lambda a$ for some $\lambda\in\RR$. That is, $\B_a=\I_a$.
\end{proof}

\begin{theorem}\label{atomic.9}
Let $X$ be an Archimedean pervasive pre-Riesz space and let $a\in \mathcal{A}_X$. Then we have $X= \B_a\oplus \B_a^\t{d}$.
\end{theorem}
\begin{proof}
Throughout the proof we use Theorem~\ref{atomic.8}.
We show that every element $x\in X$ has a unique disjoint decomposition.

\underline{Case 1}: Let $x\in X_+$.

\underline{Case 1.1}: Let $0<a\leq x$. By Lemma~\ref{atomic.7}~\ref{atomic.7.eq1} there exists a real $\lambda >0$ such that $x-\lambda a\geq 0$ and $(x-\lambda a)\perp a$. 
It follows $(x-\lambda a)\perp \lambda a$. That is, $x$ has a disjoint decomposition into positive elements
\begin{equation}\label{atomic.9.eq1}
x = \lambda a + (x-\lambda a) \quad\t{ with } \lambda a \in\left\{a\right\}\dd \t{ and } (x-\lambda a)\in\left\{a\right\}^\t{d}.
\end{equation}
By Lemma~\ref{atomic.7}~\ref{atomic.7.eq2} this decomposition is unique.

\underline{Case 1.2}: Let $x\in X_+$ be arbitrary. If $x\perp a$, then $x\in\left\{a\right\}^\t{d}$ and thus $x$ has the unique disjoint decomposition $x = 0 + x$, where $0\in\left\{a\right\}\dd$ and $x\in\left\{a\right\}^\t{d}$. If $x\not\perp a$, then $i(x)\Inf i(a) \neq 0$. Thus we have $0<i(x)\Inf i(a) \leq i(a)$. Since $a$ is an atom in $X$, by Theorem~\ref{atomic.4}~\ref{atomic.4.eq1} the element $i(a)$ is an atom in $Y$. Hence there exists a real $\lambda >0$ such that $i(x)\Inf i(a) = \lambda i(a)$. By Proposition~\ref{properties.23y}, the infimum $x\Inf a$ exists in $X$, and we have $x\Inf a = \lambda a$ in $X$.
The element $\tilde{a}:=\lambda a$ is an atom in $X$ and we obtain $0< x\Inf a = \tilde{a} \leq x$. By Case 1.1 there exists a unique disjoint decomposition into positive elements as in \eqref{atomic.9.eq1}, i.e.\ there exists $\mu\in\RR$ such that $x = \mu \tilde{a} + (x-\mu \tilde{a})$. Since for some real $\gamma >0$ we have $\mu \tilde{a} = \gamma a$, the element $x$ has a unique disjoint decomposition
$x = \gamma a + (x-\gamma a)$.

\underline{Case 2}: Let $x\in X$. Due to Case 1 we conclude that, since $x$ can be represented as a difference of two positive elements, it has a disjoint decomposition
\begin{equation}\label{disjontdecomp}
x= x_a + x_d \quad\t{ with } x_a\in\left\{a\right\}\dd=\B_a,\hs x_d\in\left\{a\right\}^\t{d},\hs x_a \perp x_d.
\end{equation}
We have to show that this decomposition is unique.

First we show that the decomposition of positive elements is compatible with addition. Indeed, for $y,z\in X_+$ the positive element $y+z$ has a unique decomposition $y+z= (y+z)_a + (y+z)_d$ with $(y+z)_a \in \B_a$ and $(y+z)_d \in \B_a^\t{d}$. Moreover, both $y$ and $z$ have a decomposition $y=y_a +y_d$ and $z = z_a+z_d$, where $y_a, z_a \in \B_a$ and $y_d, z_d \in \B_a^\t{d}$. Therefore the element $y+z$ has a second decomposition $y+z = y_a+z_a + y_d + z_d$. Since the decompositions for positive elements are unique, it follows that the parts of $y+z$ which belong to $\B_a$ coincide, i.e.\ $(y+z)_a = y_a + z_a$, and analogously $(y+z)_d = y_d + z_d$. We conclude that the decomposition of positive elements is compatible with addition, i.e.\ for $y,z \in X_+$ we have
\begin{equation}\label{atomic.9.eq2}
(y+z)_a = y_a + z_a \quad\t{ and }\quad (y+z)_d = y_d + z_d.
\end{equation}

Finally, we establish the uniqueness of the disjoint decomposition in \eqref{disjontdecomp}.
Let $x\in X$. Since $X$ is directed, $x$ can be written as a difference of two positive elements.

\underline{Case 2.1}: Let $x=x^{(1)}-x^{(2)}=x^{(3)}-x^{(4)}$  with $x^{(1)}, x^{(2)}, x^{(3)}, x^{(4)} \geq 0$ and assume that $x^{(1)}$ and $x^{(3)}$ are comparable, e.g.\ $x^{(1)}>x^{(3)}$. Then with $\tilde{x}:=x^{(1)}-x^{(3)}>0$ we have
\begin{align*}
x^{(2)} \hs &=\hs x^{(1)} - x^{(1)} + x^{(2)} \hs=\hs x^{(1)} -x \\
\hs &=\hs (x^{(1)} -\tilde{x}) - x +\tilde{x} \hs=\hs x^{(3)} -x +\tilde{x} = x^{(4)} + \tilde{x}.
\end{align*}
This and the definition of $\tilde{x}$ yield
\begin{equation}\label{atomic.9.eq3}
\quad x^{(1)} = x^{(3)} + \tilde{x} \quad\t{ and }\quad x^{(2)} = x^{(4)} + \tilde{x}.
\end{equation}
Moreover, for the positive elements $x^{(1)}, x^{(2)}, x^{(3)}, x^{(4)}$ and $\tilde{x}$ by Case 1 we have unique disjoint decompositions, i.e.\ $x^{(1)}= x^{(1)}_a + x^{(1)}_d$ with $x^{(1)}_a\in \B_a$ and $x^{(1)}_d\in \B_a^\t{d}$, etc. Since the decomposition is compatible with addition, as in \eqref{atomic.9.eq2}, from \eqref{atomic.9.eq3} it follows
\[ x^{(1)}_a - x^{(2)}_a \hs=\hs (x^{(3)} + \tilde{x})_a - (x^{(4)} + \tilde{x})_a \hs=\hs x^{(3)}_a + \tilde{x}_a - x^{(4)}_a - \tilde{x}_a \hs=\hs x^{(3)}_a-x^{(4)}_a,\]
i.e.\ $x_a:=x^{(1)}_a - x^{(2)}_a=x^{(3)}_a-x^{(4)}_a$ is independent of the choice of representatives. Analogously it follows $x_d:=x^{(1)}_d-x^{(2)}_d = x^{(3)}_d-x^{(4)}_d$. Thus the element $x=x_a+x_d$ has a unique disjoint decomposition with $x_a\in\B_a$ and $x_d\in \B_a^\t{d}$.

\underline{Case 2.2}: Let there be two decompositions of $x\in X$ as $x=x^{(3)}-x^{(4)} = x^{(5)}-x^{(6)}$, with $x^{(3)}$ and $x^{(5)}$ not comparable. Since $X$ is directed, there exists an element $x^{(1)}\in X$ such that $x^{(1)} \geq x^{(3)},x^{(5)}, x$. The element $x$ has a decomposition $x=x^{(1)}-(x^{(1)}+x)$ with $x^{(1)}\geq 0$ and $x^{(2)}:=x^{(1)}-x\geq 0$. Moreover, we have $x^{(1)}>x^{(3)}$. By Case 2.1 the disjoint decomposition of $x$ as $x=x_a+x_d$ is independent of the choice of representatives with respect to $x=x^{(1)}-x^{(2)} = x^{(3)}-x^{(4)}$. Similarly, due to $x^{(1)}>x^{(5)}$ it is independent of the choice of representatives with respect to $x=x^{(1)}-x^{(2)} = x^{(5)}-x^{(6)}$. Hence the decomposition is independent with respect to $x=x^{(3)}-x^{(4)} = x^{(5)}-x^{(6)}$.

Altogether, we established that for every $x\in X$ there exists a unique (disjoint) decomposition
\begin{equation*}
x = x_a +x_d \quad\t{ with } x_a\in\B_a \t{ and } x_d\in\B_a^\t{d},
\end{equation*}
that is, $X = \B_a\oplus \B_a^\t{d}$.
\end{proof}

As a consequence of Theorem~\ref{Glueck}\ref{Glueck.it1} and Theorem~\ref{atomic.9} we obtain the following.
\begin{corollary}\label{atomic.10}
Let $X$ be an Archimedean pervasive pre-Riesz space and $a\in \mathcal{A}_X$. Then there is an order projection $P\colon X\rightarrow X$ with range $\B_a$.
\end{corollary}

The following example demonstrates that the conclusion of Corollary~\ref{atomic.10} is not true, in general, if the pre-Riesz space is not pervasive.
\begin{example}\label{projection_on_ideal}
\emph{An Archimedean pre-Riesz space with atoms and only trivial order projections.}

We return to Example~\ref{atomic.4u}, where we considered the pre-Riesz space $X = \RR^3$ endowed with the four-ray cone $K_4$. Recall that $\mathcal{A}_X=\left\{\lambda v_k \mid| \lambda\in\RR_{>0} \t{ and } v_k\in\left\{v_1,\ldots,v_4\right\}\right\}$ 
and that by Theorem~\ref{Glueck}\ref{Glueck.it1} order projections and band projections on $X$ coincide. Due to Theorem~\ref{Glueck}\ref{Glueck.it2} it is sufficient to consider only projections onto directed bands.
In \c[Example~4.6]{1} it is shown that $X$ has precisely four non-trivial directed bands. They are given by $B_1:=\left\{v_1\right\}\dd \hspace*{-.5mm}=\hspace*{-.2mm} \Span\left\{v_1\right\}$, $B_2:=\Span\left\{v_2\right\}$, $B_3:=\Span\left\{v_3\right\}$ and $B_4:=\Span\left\{v_4\right\}$. 
We show that the projections onto the bands $B_1$ to $B_4$ are trivial. First, let $P$ be a projection onto $B_1$. For every $k\in\left\{2,3,4\right\}$ we have
\begin{equation}\label{projection_on_ideal.eq1}
[0,v_k]=K\cap(v_k-K)=\left\{x\in B_k\mid| 0\leq x\leq v_k\right\}\sub B_k.
\end{equation}
Due to $P(X)\sub B_1$ and $0\leq P\leq I$, where $I$ is the identity map on $X$, the equation \eqref{projection_on_ideal.eq1} implies for $k\in\left\{2,3,4\right\}$ that $P(v_k)\in B_1\cap[0,v_k]\sub B_1\cap B_k =\left\{0\right\}$, i.e.\ $P(v_k)=0$. As the vectors $v_2,v_3$ and $v_4$ form a base of $X$, it follows $P=0$. Analogously, the band projection onto any of the bands $B_2,B_3$ and $B_4$ is trivial. We conclude that $X$ has only the trivial order projections $0$ and $I$.
\end{example}

To sum up, we obtain for Archimedean pervasive pre-Riesz spaces the extension property for projection bands and a theory of bands generated by atoms similar to the vector lattice case. As one can see in Theorems~\ref{extideal_extband_coincide}, \ref{atomic.4} and \ref{atomic.8}, pervasiveness turns out to be the key condition in these results.


\bibliographystyle{plain}

\end{document}